\documentclass[a4paper,11pt]{article}

\usepackage{graphicx,enumitem}
\usepackage{color}
\usepackage{latexsym}
\usepackage{amsmath, amssymb, mathrsfs}
\usepackage{bm}
\usepackage{algorithmic}
\usepackage{algorithm}

\usepackage{soul,xcolor}
\usepackage[colorlinks]{hyperref}
\usepackage{tikz}
\usetikzlibrary{arrows.meta}
\usetikzlibrary{calc}
\usepackage{subcaption} 
\usepackage{amsfonts,amsmath,amssymb}
\usepackage{amsthm}
\usepackage{geometry}
\usepackage{graphicx} 
\usepackage{xcolor}
\usepackage{hyperref}
\usepackage{footnote}
\usepackage{stackrel}
\usepackage{subcaption}

\usepackage[tikz]{mdframed}
\newenvironment{itembox}[1]{\begin{mdframed}[roundcorner=10pt,
  frametitlefont=\normalfont,
  frametitleaboveskip=\dimexpr-0.7\baselineskip,skipabove=\topskip,
  innertopmargin=\dimexpr-0.65\baselineskip,
  innerbottommargin=\dimexpr0.65\baselineskip,
  frametitle={\tikz{\node[anchor=base,rectangle,fill=white] {\strut #1};}}]
}{\end{mdframed}}

\numberwithin{equation}{section} 
\newtheorem{theorem}{Theorem}[section]
\newtheorem{dfn}{Definition}

\newtheorem{prob}{Problem}
\newtheorem{numex}{Numerical Example}
\newtheorem{remark}{Remark}
\newtheorem{proposition}{Proposition}
\newtheorem{lemma}[theorem]{Lemma}

\newcommand{\bb}[1]{\boldsymbol{{#1}}}

\newcommand{\ddn}[1]{\frac{\partial{#1}}{\partial {\nu}}}

 \usepackage{mathpazo}

\title{Comoving mesh method for certain classes of moving boundary problems\footnotetext{The work of MK was supported by JSPS KAKENHI Grant Number JP20KK0058. JFTR acknowledges the support from JST CREST Grant Number JPMJCR2014.}}
\author{Yosuke Sunayama$^{\ast}$ \qquad Masato Kimura$^\dagger$ \qquad Julius Fergy T. Rabago$^\ddagger$}

\date{%
	{\footnotesize
	$^{\ast}$Division of Mathematical and Physical Sciences,\\%
         Kanazawa University, Kanazawa 920-1192, Japan\\%
	\texttt{sunayama\_math@stu.kanazawa-u.ac.jp}\\[2ex]%
	$^\dagger$Faculty of Mathematics and Physics,\\%
         Kanazawa University, Kanazawa 920-1192, Japan\\%
	\texttt{mkimura@se.kanazawa-u.ac.jp}\\[2ex]%
	$^\ddagger$Faculty of Mathematics and Physics,\\%
         Kanazawa University, Kanazawa 920-1192, Japan\\\vspace{-2pt}
        \texttt{rabagojft@se.kanazawa-u.ac.jp,\ jfrabago@gmail.com}}\\[2ex]
    \today
}

\begin{document}

\maketitle


\begin{abstract}
  A Lagrangian-type numerical scheme called the ``comoving mesh method'' or CMM is developed for numerically solving certain classes of moving boundary problems which include, for example, the classical Hele-Shaw flow problem and the well-known mean curvature flow problem.
  This finite element scheme exploits the idea that the normal velocity field of the moving boundary can be extended throughout the entire domain of definition of the problem using, for instance, the Laplace operator.
  Then, the boundary as well as the finite element mesh of the domain are easily updated at every time step by moving the nodal points along this velocity field.
  The feasibility of the method, highlighting its practicality, is illustrated through various numerical experiments.
  Also, in order to examine the accuracy of the proposed scheme, the experimental order of convergences between the numerical and manufactured solutions for these examples are also calculated.\\


\end{abstract}


\section{Introduction}
\label{intro}
We are interested in the numerical approximation of solutions to certain classes of moving boundary problems for $d$-dimensional ($d = 2,3$) bounded domains that includes, specifically, the so-called single phase Hele-Shaw problem.
The classical Hele-Shaw moving boundary problem seeks a solution to a Laplace's equation in an unknown region whose boundary changes with time.
In the present study, we are actually interested with the more general Hele-Shaw problem that also arises in shape optimization problems.

Let $T>0$ be fixed and $B$ be an open bounded set in $\mathbb{R}^d$ $(d=2,3)$ with a smooth boundary $\partial B$.
For $t\in[0,T]$, consider a larger open bounded set $\Omega(t) \subset \mathbb{R}^d$ containing $\overline{B}$ with boundary $\Gamma(t):={\partial\Omega}(t)$ such that $\partial B \cap \Gamma(0) = \emptyset$ (i.e., $\operatorname{dist}(\partial B, \Gamma(0)) > 0$).
Denote by ${\nu}$ the outward unit normal vector on the boundary of $\Omega(t) \setminus \overline{B}$ as illustrated in Fig. \ref{fig:Fig1}.
Given the functions $f:\mathbb{R}^d \times [0,T] \to \mathbb{R}$, $q_B:\partial B \times [0,T] \to \mathbb{R}$, $\bb{\gamma}:\mathbb{R}^d \times [0,T] \to \mathbb{R}^d$, the constant $\lambda \in \mathbb{R}$, and the initial profile $\Omega_0$ of $\Omega(t)$, with $V_{n}:=V_{n}(x,t)$, $x \in \Gamma(t)$, describing the outward normal velocity of the moving interface $\Gamma(t)$,
we consider the following moving boundary problem:
\begin{prob}
\label{prob:general_Hele-Shaw}
Find $\Omega(t) \supset \overline{B}$ and $u(\cdot,\ t):\overline{\Omega(t)} \setminus B \to \mathbb{R}$ such that
	\begin{equation}
	\label{eq:general_Hele-Shaw}
	\left\{\arraycolsep=1.4pt\def\arraystretch{1}
	\begin{array}{rcll}
		-\Delta u					&=&f 		&\quad\text{in $\Omega(t)\setminus \overline{B}$, \quad $t \in [0,T]$},\\
		(1-\alpha)u + \alpha \nabla u \cdot \nu
								&=&q_{B}		&\quad\text{on $\partial B$},\\
		u						&=&0		&\quad\text{on $\Gamma(t)$, \quad $t \in [0,T]$},\\
		V_{n} 					&=&(- \nabla u + \bb{\gamma})\cdot {\nu} +\lambda	&\quad\text{on $\Gamma(t)$, \quad $t \in [0,T]$},\\
		\Omega(0)				&=& \Omega_0,
	\end{array}
	\right.
	\end{equation}
	where $\alpha \in \{0,1\}$.
\end{prob}
Here, for simplicity, we assume that the boundaries $\partial B$ and $\Gamma(t)$ are smooth, or equivalently, of class $C^{\infty}$.
The topological situation illustrating the above problem is depicted in Fig. \ref{fig:Fig1}.
\begin{figure}[htbp]
\centering
	\scalebox{0.33}{\includegraphics{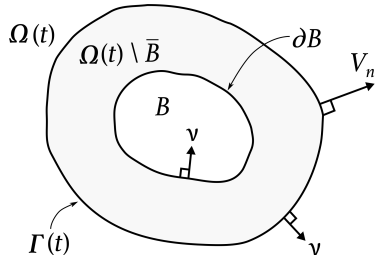}}
	\caption{The moving domain $\Omega(t)$ and fixed domain $B$}
	 \label{fig:Fig1}       
\end{figure}
In \eqref{eq:general_Hele-Shaw}, the parameter $\alpha$ indicates whether the boundary condition on the fixed boundary $\partial B$ is a Dirichlet boundary condition $(\alpha = 0)$ or a Neumann boundary condition ($\alpha = 1$).
The fourth equation in \eqref{eq:general_Hele-Shaw} expresses the motion of the free boundary that evolves according to $V_{n} = (- \nabla u + \bb{\gamma})\cdot {\nu}+\lambda$,
where the function $u$ satisfies the first three equations in \eqref{eq:general_Hele-Shaw}.
Here, equation \eqref{eq:general_Hele-Shaw} with $f \equiv 0$, $\bb{\gamma} \equiv \bb{0}$ and $\lambda = 0$ is also known in the literature as the classical Hele-Shaw problem or simply the Hele-Shaw problem (see, e.g., \cite{EscherSimonnet1997}):
Let us discuss more about the case $f \equiv 0$ and $\lambda = 0$ in \eqref{eq:general_Hele-Shaw}.
If $\alpha = 1$, $q_{B} > 0$, and $\bb{\gamma} \equiv 0$, problem \eqref{eq:general_Hele-Shaw} describes a model of the expanding (two-dimensional) Hele-Shaw flow (see, e.g., \cite{Crank1984,ElliottOckendon1982,Elliott1980,ElliottJanovsky1981,Richardson1972}) which provides a simple description either of the flow of a viscous Newtonian liquid between two horizontal plates separated by a thin gap, or of a viscous liquid moving under Darcy's law in a porous medium \cite{CummingsHowisonKing1999} (see also \cite{MilneThomsonBook1996}).
In a typical situation, $u$ represents the pressure in an incompressible viscous fluid blob $\Omega \setminus \overline{B}$, and
because the Neumann flux $q_B$ is positive, more fluid is injected through the fixed boundary $\partial B$.
As a result, the blob expands in time and is modelled by the moving boundary $\Gamma$.
The problem is sometimes formulated with the prescribed pressure (i.e., $q_{B}$ is now interpreted as a given pressure instead of a Neumann flux) on the fixed boundary, i.e., with the non-homogeneous Dirichlet boundary condition $u = q_{B}$ on $\partial B$ (see, e.g., \cite{FasanoPrimecerio1993}).
This situation corresponds to the case $\alpha = 0$ in \eqref{eq:general_Hele-Shaw}.
For further classical applications of \eqref{eq:general_Hele-Shaw} at the current setting, we refer the readers, for example, to \cite{Crank1984,Elliott1980,ElliottJanovsky1981,Friedman1979}.
In the case that $\bb{\gamma} \not\equiv 0$, the given quantity may, in a sense, be interpreted as an (external) background flow.
Here, we do not consider the interesting question of existence of unique \emph{classical} solution to the general problem \eqref{eq:general_Hele-Shaw}, but readers may refer to \cite{EscherSimonnet1997} for existence result in the case of $f \equiv 0$ and $q_{B} > 0$.
Nevertheless, this issue will be the topic of our future investigation.
Meanwhile, results regarding existence of a weak solution to the Hele-Shaw problem via variational inequalities can be found in \cite{Elliott1980,ElliottJanovsky1981,Gustafsson1985}.
Of course, it would be nice if we could actually transform equation \eqref{eq:general_Hele-Shaw} into an elliptic variational inequality formulation such as in the case of the classical Hele-Shaw problem (see \cite{ElliottJanovsky1981}).
However, it seems that such method which employs the so-called Baiocchi transform \cite{BaiocchiBook1984} does not apply directly to our problem due to the presence of the external background flow $\bb{\gamma}$.
Moreover, we emphasize that we are not aware of any existing solution methods to treat the given problem.
So, as in many past studies, this motivates us to at least find an approximate numerical solution to the problem for concrete cases by providing a simple and convenient numerical method in accomplishing the task.

Problem \ref{prob:general_Hele-Shaw} is also related to the Bernoulli free boundary problem.
Suppose now that $f=f(x)$, $\bb{\gamma} \equiv 0$, and $\lambda < 0$, and that the shape solution to \eqref{eq:general_Hele-Shaw} happens to converge to a stationary point as $t$ increases indefinitely, i.e., there exists a domain $\Omega^*$ such that $ V_{n} = 0$ on $\Gamma^* = \partial \Omega^*$, then we call \eqref{eq:general_Hele-Shaw} a generalized \emph{exterior} Bernoulli-like free boundary problem:
\begin{prob}
Given a negative constant $\lambda$ and a fixed open bounded set $B$, find a bounded domain $\Omega \supset \overline{B}$ and a function $u:\overline{\Omega} \setminus B \rightarrow \mathbb{R}$ such that
	\begin{equation}
	\label{eq:generallized_Bernoulli_problem}
	\left\{\arraycolsep=1.4pt\def\arraystretch{1}
	\begin{array}{rcll}
		-\Delta u		&=&f 						&\quad\text{in $\Omega\setminus \overline{B}$},\\
		(1-\alpha)u + \alpha \nabla u \cdot \nu
					&=&q_{B}	&\quad\text{on $\partial B$},\\
		u\ =\ 0 \quad \text{and}\quad \nabla u\cdot{\nu}&=& \lambda						&\quad\text{on $\Gamma$}.
	\end{array}
	\right.
	\end{equation}
\end{prob}

Bernoulli problems find their origin in the description of free surface for ideal fluids (see, e.g., \cite{Friedman1984,Friedrichs1934}).
However, it also arises in the context of optimal design, such as in electro chemical machining and galvanization \cite{Crank1984,LaceyShillor1987}, as well as in insulation problems \cite{Acker1981,Flucher1993}.
For some qualitative properties of solutions to the Bernoulli problem, including existence, classifications, and uniqueness of its solution, and some ideas about numerical approximations of its solutions via fixed-point iterations, we refer the readers to \cite{FlucherRumpf1997}, as well as to the references therein (see also \cite{RabagoThesis2020}).

As mentioned earlier, our main objective in this study is to present a simple numerical scheme for solving the moving boundary problem \eqref{eq:general_Hele-Shaw}.
Of course, there are already several numerical approaches to solve the present problem, especially in the case of the Hele-Shaw flow $V_{n} = -\nabla u \cdot {\nu}$ (with $f \equiv 0$, $\bb{\gamma} \equiv \bb{0}$, $\lambda = 0$, $\alpha = 1$, and $q_{B} > 0$ in \eqref{eq:general_Hele-Shaw}).
In fact, it is well-known that the Hele-Shaw problem can be solved numerically using the boundary element method which was employed, for instance, in \cite{GustafssonVasilev2006}, or the charge simulation method (CSM) applied in \cite{Kimura1997,SakakibaraYazaki2015}.
The latter method can also be used to other two-dimensional moving boundary problems, but is not actually easy to utilized in the case of three-dimensional problems.
To address this difficulty, the authors in \cite{KimuraNotsu2002} proposed an improvement of CSM by combining it with the level-set method.
Still, however, to the best of our knowledge, no convenient and effective numerical approach has yet been developed to numerically solve the more general equation \eqref{eq:general_Hele-Shaw} with $f \not\equiv 0$ and $\bb{\gamma} \not\equiv \bb{0}$.
The purpose of this investigation, therefore, is to fill this gap by developing a numerical method to solve \eqref{eq:general_Hele-Shaw} with the following three main characteristics:
\begin{itemize}[label=$\bullet$]
	\item firstly, as opposed to CSM, our proposed method is easier to implement and can easily treat three-dimensional moving boundary problems without ad hoc procedures;
	\item secondly, in contrast to existing traditional finite element methods used to solve many moving boundary problems, our propose scheme does not require mesh regeneration at every time step in the approximation process;
	\item and, lastly, our method can easily be adapted to solve other classes of moving boundary problems, such as the mean curvature problem.
\end{itemize}
The rest of the paper is organized as follows.
In Section \ref{sec:CMM}, we formally introduce and give the motivation behind our proposed method which we termed as the `comoving mesh method'.
We also write out the structure of the numerical algorithm for the method, and then illustrate its applicability in solving the Hele-Shaw problem.
Moreover, we evaluate the correctness and accuracy of the scheme through the method of manufactured solution.
Then, in Section \ref{sec:Bernoulli}, we will discuss how equation \eqref{eq:general_Hele-Shaw} is closely related to the so-called exterior Bernoulli problem in connection with a shape optimization formulation of the said free boundary problem (FBP).
In addition, we numerically solve the FBP using our propose scheme.
Meanwhile, in Section \ref{sec:MCF}, as further application of CMM, we will also apply our method to curve shortening problem, thus showcasing the versatility of the method.
Furthermore, in Section \ref{sec:properties}, we state and prove two simple qualitative properties of the proposed numerical approximation procedure.
Finally, we end the paper by giving out a concluding statement in Section \ref{sec:conclusion} and a brief remark about our future work.

\section{The Comoving Mesh Method for the Hele-Shaw Problem}
\label{sec:CMM}
This section is mainly devoted to the introduction of the proposed method.
The motivations behind its formulation are also given in this section.
Moreover, the structure of the algorithm that will be used in the numerical implementation of the method is also provided here.
This is followed by a presentation of two simple numerical examples illustrating the applicability of the scheme in solving concrete cases of problem \eqref{eq:general_Hele-Shaw}, one with $\bb{\gamma}(\cdot,t) \equiv \bb{0}$ and the other one with $\bb{\gamma}(\cdot,t) \not \equiv \text{constant} \neq 0$ on $\Gamma(t)$ ($t \in [0,T]$).
To check the accuracy of the proposed scheme, we also examine the error of convergence or EOC of the method with the help of the method of manufactured solutions \cite{SalariKnupp2000}.

\subsection{Idea and motivation behind CMM}
\label{sec:idea}

As alluded in Introduction, the main purpose of the present paper is the development of a simple Lagrangian-type numerical scheme that we call ``comoving mesh method,'' or simply CMM, for solving a class of moving boundary problems.
To begin with, we give out a naive idea of the method.
For simplicity, we set $\alpha = 1$.
Let $T>0$ be a given final time of interest, $N_T$ be a fixed positive integer, and set the time discretization step-size as $\tau := T/N_T$.
For each time-step index $k = 0,1,\cdots,N_T$, we denote the time discretized domain by $\Omega^k \approx \Omega(k \tau)$ (similarly, $\Gamma^k \approx \Gamma(k \tau)$) and the associated time discretized function as $u^k \approx u(\cdot, k \tau)$, $f^k \approx f(\cdot, k \tau)$, $q_B^k \approx q_B(\cdot, k \tau)$, and $\bb{\gamma}^k \approx \bb{\gamma}(\cdot, k \tau)$.
The rest of the notations used below are standard and will only be stressed out for clarity.

After specifying the final time of interest $T>0$ and deciding the value of $N_T \in \mathbb{N}$,
a naive numerical method for the Hele-Shaw problem \eqref{eq:general_Hele-Shaw} consists of the following three steps:
\begin{itembox}{{Conventional scheme for \eqref{eq:general_Hele-Shaw}}} 
\vspace{5pt}
At each time $t = k \tau$:
\begin{description}
	\item{\underline{Step 1.}} The first step is to solve $u^k$ over the domain $\Omega^k \setminus \overline{B}$:
	\[
		-\Delta u^k 	=	f^k 		\ \ \text{in $\Omega^k\setminus \overline{B}$},\quad
		\nabla u^k \cdot \nu^k =q_{B}^k	\ \ \text{on $\partial B$},\quad
		u^k			=	0		\ \ \text{on $\Gamma^k$}.
	\]
	\item{\underline{Step 2.}} Then, we define the normal velocity of $\Gamma^k$ in terms of the function $u^k$ and the normal vector ${\nu}^k$ to $\Gamma^k$, i.e., we set $V_{n}^k := (-\nabla u^k + \bb{\gamma}^k) \cdot \nu ^k + \lambda$ on $\Gamma ^k$.
	\item{\underline{Step 3.}} Finally, we move the boundary along the direction of the velocity field $V_{n}^k$, i.e., we update the moving boundary according to \[\Gamma^{k+1} := \left\{ x+\tau V_{n}^k (x) \nu^k (x) \ \middle\vert \ x \in \Gamma^k \right\}.\]
\end{description}
\end{itembox}

However, there are two obstacles in the realization of this naive idea in a finite element method (FEM).
The first main difficulty is that if $u^k$ is a piecewise linear function on a triangular finite element mesh, then $V_{n}^k$ only lives in the space $P_0(\Gamma^k_h)$ (here, of course, $\Gamma^k_h$ denotes the exterior boundary of the triangulation $\Omega_h^k \setminus \overline{B_h}$ of the domain $\Omega \setminus \overline{B}$ with the maximum mesh size $h>0$, at the current time step $k$).
Unfortunately, this local finite element space is not enough to uniquely define $V_{n}^k$ on nodal points of the mesh, and, in fact, it must belong to the (conforming piecewise) linear finite element space $P_1(\Gamma^k_h)$ in the third step.
The second one is not actually an impediment in implementing the method, but more of a preference issue in relation to mesh generation.
Typically, moving boundary problems requires mesh regeneration when solved using finite element methods; that is, one needs to generate a triangulation of the domain $\Omega^k \setminus \overline{B}$ at each time step $k$ after the boundary moves.
To circumvent these issues, we offer the following remedies.

We first address the second issue.
In order to avoid generating a triangulation of the domain at every time step, we move not only the boundary, but also the internal nodes of the mesh triangulation at every time step.
By doing so, the mesh only needs to be generated at the initial time step $k=0$.
This is the main reason behind the terminology used to name the present method (i.e., the `comoving mesh' method).
In order to move the boundary and internal nodes simultaneously, we first create a smooth extension ${\bb{w}^k}$ of the velocity field $V_{n}^k{\nu}^k$ into the entire domain $\Omega^k \setminus \overline{B}$ using the Laplace operator.
This is done more precisely by finding ${\bb{w}^k_h} \in P_1(\Omega_h^k\setminus \overline{B_h};\mathbb{R}^d)$ which is a finite element solution to the following Laplace equation:
	\begin{equation}
	\label{eq:naive}
	- \Delta {\bb{w}^k} 	=  \bb{0} \ \ \text{in $\Omega_h^k \setminus \overline{B_h}$},\qquad
		{\bb{w}^k} 	=  \bb{0} \ \ \text{on $\partial B_h$},\qquad
 		{\bb{w}^k}		= V_{n}^k {\nu}^k \ \ \text{on $\Gamma_h^k$},
	\end{equation}
where, we suppose a polygonal domain $\overline{\Omega_h^k}$, at $t=k \tau$, and its triangular mesh $\mathcal{T}_h(\overline{\Omega_h^k}\setminus B_h) = \{ K^k_l \} ^{N_e}_{l=1}$ ($K^{k}_l$ is a closed triangle $(d=2)$, or a closed tetrahedron $(d=3)$), are given,
 and $P_1(\Omega_h^k\setminus \overline{B_h};\mathbb{R}^d)$ denotes the $\mathbb{R}^d$-valued piecewise linear function space on $\mathcal{T}_h(\overline{\Omega_h^k}\setminus B_h)$.
Then, $\Omega_h^{k+1}$ and $\mathcal{T}_h(\overline{\Omega_h^{k+1}}\setminus B_h) = \{ K^{k+1}_l \} ^{N_e}_{l=1}$ are defined as follows:
\begin{eqnarray}
  \label{eq:mesh-update}
	\overline{\Omega^{k+1}_h}\setminus B_h
		&:=&\left \{ x + \tau \bb{w}^k_h(x) \ \middle\vert \
			x \in \overline{\Omega^{k}_h}\setminus B_h \right\}, \\
  \label{eq:triangle-update}
  K^{k+1}_l
    &:=&\left\{ x + \tau \bb{w}^k_h(x) \ \middle\vert \
    	x \in K^{k}_l \right\},
\end{eqnarray}
for all $k = 0,1,\cdots,N_T$, see Fig. \ref{fig:Fig2} for illustration.

\begin{remark}
  If ${\bb{w}^k_h}$ is belongs to $P_2$ or higher order finite element space, then, instead of \eqref{eq:triangle-update}, we set the triangular mesh $\mathcal{T}_h(\overline{\Omega_h^k}\setminus B_h)$ with the set of nodal points $\mathcal{N}^k_h = \{ p^k_j \} ^{N_p}_{j=1}$ :
  \begin{equation}
  \label{eq:node-update}
  \mathcal{T}_h(\overline{\Omega_h^{k+1}}\setminus B_h) :=
  \left\{\arraycolsep=1.4pt\def\arraystretch{1.5}
  \begin{array}{rcll}
    p^{k+1}_j &:=& p^k_j + \tau {\bb{w}^k_h}(p^k_j) \\[0.2em]
    K^{k+1}_l \cap \mathcal{N}^{k+1}_l &=& K^k_l \cap \mathcal{N}^k_l	.
  \end{array}
  \right.
  \end{equation}
\end{remark}
Note that the definition of the (discrete) time evolution of the annular domain $\overline{\Omega} \setminus B$ given in \eqref{eq:mesh-update} clearly agrees with the original characteristic (at least for the interior boundary) of its desired evolution.
This is because the choice of extension for the vector field $V_{n}{\nu}$ fixes the boundary $\partial B$ of the interior domain $B$ throughout the entire time evolution interval $[0,T]$.
It is worth to emphasize that a similar idea is adopted in the so-called \emph{traction method} developed by Azegami \cite{Azegami1994} for shape optimization problems (see also \cite{AzegamiBook2020}).
\begin{figure}[htbp]
\centering
        \begin{subfigure}[b]{0.55\textwidth}
                \centering
               \raisebox{-0.5\height}{\resizebox{\textwidth}{!}{\includegraphics{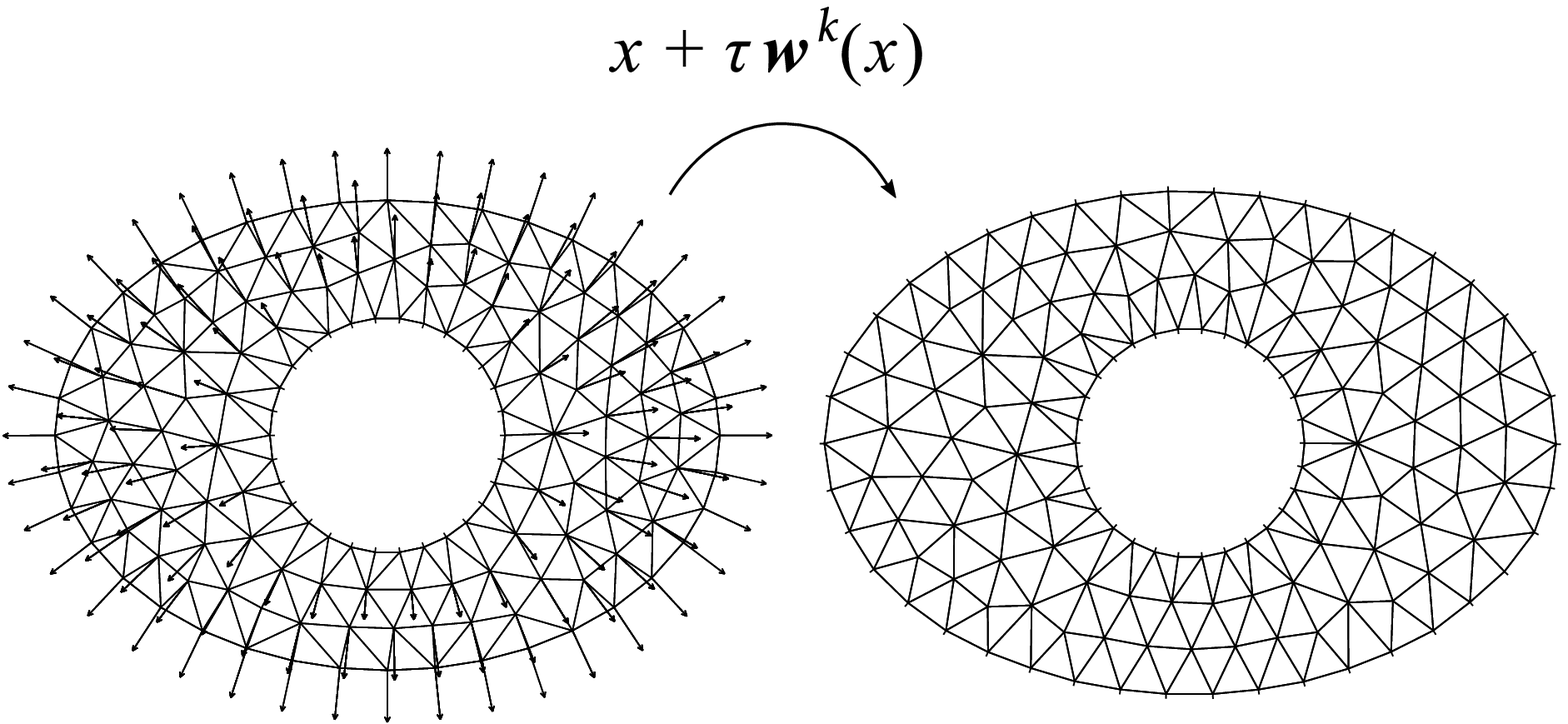}}}
                \caption{Nodal points relocation}
                \label{fig:Fig2a}
        \end{subfigure}%
        \hfill
	 \begin{subfigure}[b]{0.45\textwidth}
                \centering
                 \raisebox{-0.5\height}{\resizebox{0.6\textwidth}{!}{\includegraphics{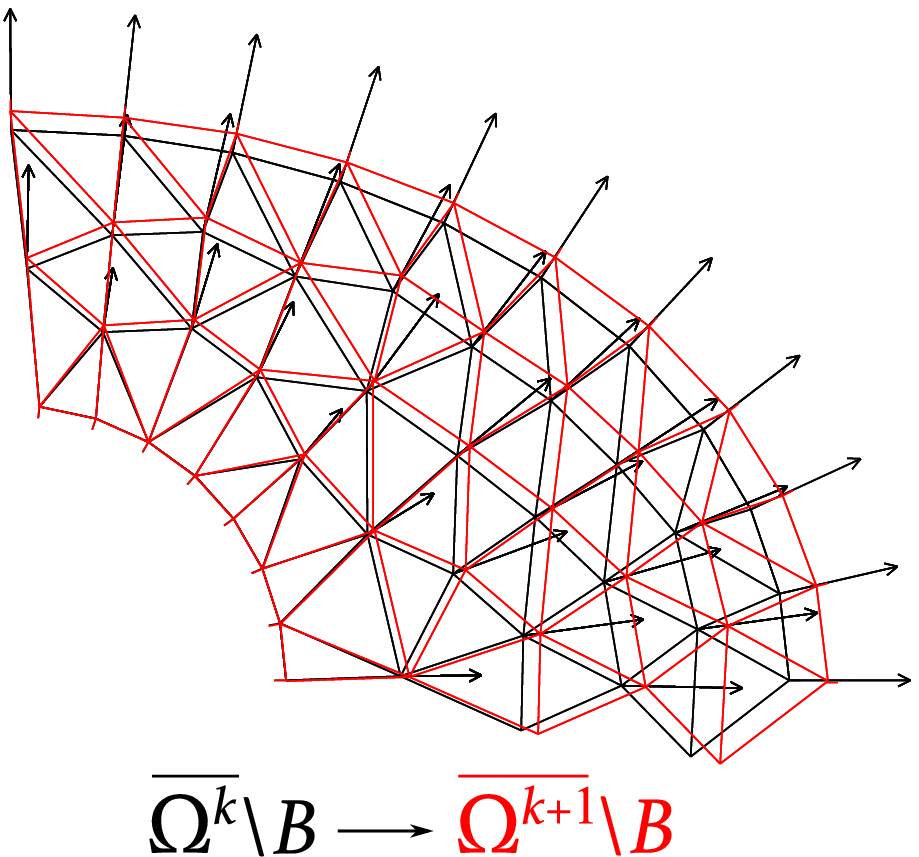}}}
                 \caption{A superimposed sectional illustration}
                \label{fig:Fig2b}
        \end{subfigure}%
\caption{Plot \ref{fig:Fig2a}: initial and deformed mesh after nodes relocation (scaled with the time-step parameter $\tau$ and moved in accordance with the direction of the velocity field $\bb{w}$); plot \ref{fig:Fig2b}: a superimposed comparison of corresponding sections of the domains $\overline{\Omega^{k}} \setminus B$ and $\overline{\Omega^{k+1}} \setminus B$}
\label{fig:Fig2}
\end{figure}

On the other hand, concerning the first issue, we treat the $P_0(\Gamma^k_h)$-function by using a Robin approximation ${\varepsilon} \nabla {\bb{w}^k} \cdot {\nu}^k + {\bb{w}^k} = V_{n}^k {\nu}^k$ of ${\bb{w}^k} = V_{n}^k {\nu}^k$ in \eqref{eq:naive}, where $\varepsilon > 0$ is a sufficiently small fixed real number.
In other words, given $\varepsilon > 0$, we define ${\bb{w}^k_h} \in P_1(\Omega_h^k\setminus \overline{B_h};\mathbb{R}^d)$ as the finite element solution to the following mixed Dirichlet-Robin boundary value problem:
	\begin{equation}
	\label{eq:velocity}
	\left\{\arraycolsep=1.4pt\def\arraystretch{1}
	\begin{array}{rcll}
		- \Delta {\bb{w}^k} 		&=& \bb{0}	&\quad \text{in $\Omega_h^k \setminus \overline{B_h}$},\\[0.3em]
		{\bb{w}^k} 			&=& \bb{0}	&\quad \text{on $\partial B_h$},\\[0.3em]
 		{\varepsilon}  \nabla \bb{w}^k \cdot \nu^k  + {\bb{w}^k} &=& V_{n}^k {\nu}^k		&\quad \text{on $\Gamma_h^k$}.
	\end{array}
	\right.
	\end{equation}
%
In variational form, the system of partial differential equations \eqref{eq:velocity} is given as follows: find ${\bb{w}^k} \in H^1_{\partial B,\bb{0}}(\Omega_h^k\setminus \overline{B_h};\mathbb{R}^d)$ such that
	\begin{align}
   	&\displaystyle \int_{\Omega_h^k \setminus \overline{B_h}} \nabla {\bb{w}^k} : \nabla \bb{\varphi} \ {\rm d}x
    		 + \frac{1}{{\varepsilon}} \int_{\Gamma_h^k}  {\bb{w}^k} \cdot \bb{\varphi}\ {\rm d}s \nonumber\\
			&\displaystyle \hspace{1in} = \frac{1}{{\varepsilon}} \int_{\Gamma_h^k} V_{n}^k {\nu}^k \cdot \bb{\varphi}\ {\rm d}s,
    				\quad \forall \bb{\varphi} \in H_{\partial B,\bb{0}}^1(\Omega_h^k\setminus \overline{B_h};\mathbb{R}^d), \label{eq:velocity_weakform}
	\end{align}
where $H_{\partial B, \bb{0}}^1(\Omega \setminus \overline{B} ;\mathbb{R}^d)$ denotes the Hilbert space $\{\bb{\varphi} \in H^1(\Omega \setminus \overline{B} ;\mathbb{R}^d) \mid \bb{\varphi} = \bb{0} \ \text{on}\ \partial B\}$.
Obviously, the integral equation in \eqref{eq:velocity_weakform} can be evaluated even for $V_{n}^k {\nu}^k \in P_0(\Gamma_h^k)$.

To summarize the above idea, we provide the following algorithm for the comoving mesh method.
\begin{algorithm}[H]
    \caption{Comoving mesh method}
    \label{alg:cmm}
    {\fontsize{9}{10}\selectfont
    \begin{algorithmic}[1]
    \STATE Specify $T>0$, $N_T\in \mathbb{N}$, $\varepsilon>0$, and set $k=0$.
    		Also, generate a finite element mesh of the initial domain $\overline{\Omega_h^0} \setminus B_h  \approx \overline{\Omega^0} \setminus B$.
    \WHILE{$k \leqslant N_T$}
    \STATE Solve the finite element solution $u^k_h \in P_1(\Omega_h^k \setminus \overline{B_h})$ for the following:
    \vspace{-1mm}
    \[
        - \Delta u^k = f^k\quad \text{in $\Omega_h^k \setminus B_h$},\qquad
      	\nabla u^k \cdot {\nu}^k = q_B^k\quad \text{on $\partial B_h$},\qquad
        u^k = 0\quad \text{on $\Gamma_h^k$}.
    \]
    \vspace{-3mm}
    \STATE Define the normal velocity as $V_{n}^k := (-\nabla u^k_h + \bb{\gamma}^k) \cdot \nu ^k +\lambda$ on $\Gamma_h^k$.
    \STATE Create an extension of $V_{n}^k \nu^k$ by solving the finite element solution ${\bb{w}^k_h} \in P_1(\Omega_h^k\setminus \overline{B_h};\mathbb{R}^d)$ for the following:
    \vspace{-1mm}
	\[
        - \Delta {\bb{w}^k} = \bb{0}\ \ \text{in $\Omega_h^k \setminus \overline{B_h}$},\quad
        {\bb{w}^k} = \bb{0}\ \ \text{on $\partial B_h$},\quad
        \varepsilon \nabla \bb{w}^k \cdot \nu^k + {\bb{w}^k} = V_{n}^k {\nu}^k\ \ \text{on $\Gamma_h^k$}.
	\]
	\vspace{-3mm}
    \STATE Update the current domain by moving the mesh according to \eqref{eq:mesh-update} and  \eqref{eq:triangle-update}.
    \STATE $k \leftarrow k+1$
    \ENDWHILE
    \end{algorithmic}}
\end{algorithm}

\subsection{Application of CMM to a classical Hele-Shaw problem}
\label{sec:HS}
In this subsection, we apply the comoving mesh method to solve two concrete examples of problem \eqref{eq:general_Hele-Shaw}.
First, let us consider the classical Hele-Shaw problem:

\begin{equation}
\label{eq:classical_Hele-Shaw}
\left\{\arraycolsep=1.4pt\def\arraystretch{1.1}
\begin{array}{rcll}
  -\Delta u					&=&0	&\quad\text{in $\Omega(t)\setminus \overline{B}$, \quad $t \in [0,T]$},\\
  \nabla u\cdot{\nu} 			&=&1	&\quad\text{on $\partial B$},\\
  u						&=&0	&\quad\text{on $\Gamma(t)$, \quad $t \in [0,T]$},\\
  V_{n} 					&=& - \nabla u\cdot{\nu} 	&\quad\text{on $\Gamma(t)$, \quad $t \in [0,T]$},\\
  \Omega(0)				&=& \Omega_0.
\end{array}
\right.
\end{equation}
Note that, because of the maximum principle and the unique continuation property \cite{HormanderBook}, $u$ is positive in $\Omega\setminus \overline{B}$.
This means that $\nabla u\cdot{\nu} < 0$ on the moving boundary, and, in this case, since the normal velocity $V_{n}$ is always positive, the hypersurface expands.

\begin{numex}
\label{example1}
	In this example, the initial profile $\Omega_0$ of the moving domain $\Omega(t)$ ($t \in [0, T]$) is given as an ellipse, as shown in Fig. \ref{fig:Fig3a} (along with its mesh triangulation), and the final time of interest is set to $T=2$.
	Algorithm \ref{alg:cmm} is executed using mesh sizes of uniform width $h \approx 0.1$ with parameter value ${\varepsilon} = 0.1$ and time step size $\tau = 0.1$.
\end{numex}
\begin{figure}[htbp]
\centering
        \begin{subfigure}[b]{0.49\textwidth}
                \centering
                \resizebox{0.9\textwidth}{!}{\includegraphics{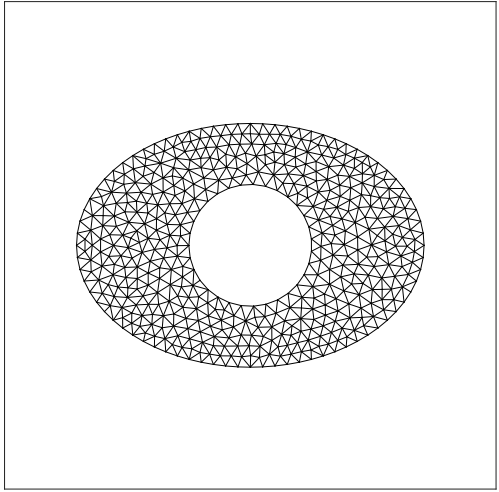}}
                \caption{Initial mesh profile $\Omega_h^0$}
                \label{fig:Fig3a}
        \end{subfigure}%
        \hfill
	 \begin{subfigure}[b]{0.49\textwidth}
                \centering
                \resizebox{0.9\textwidth}{!}{\includegraphics{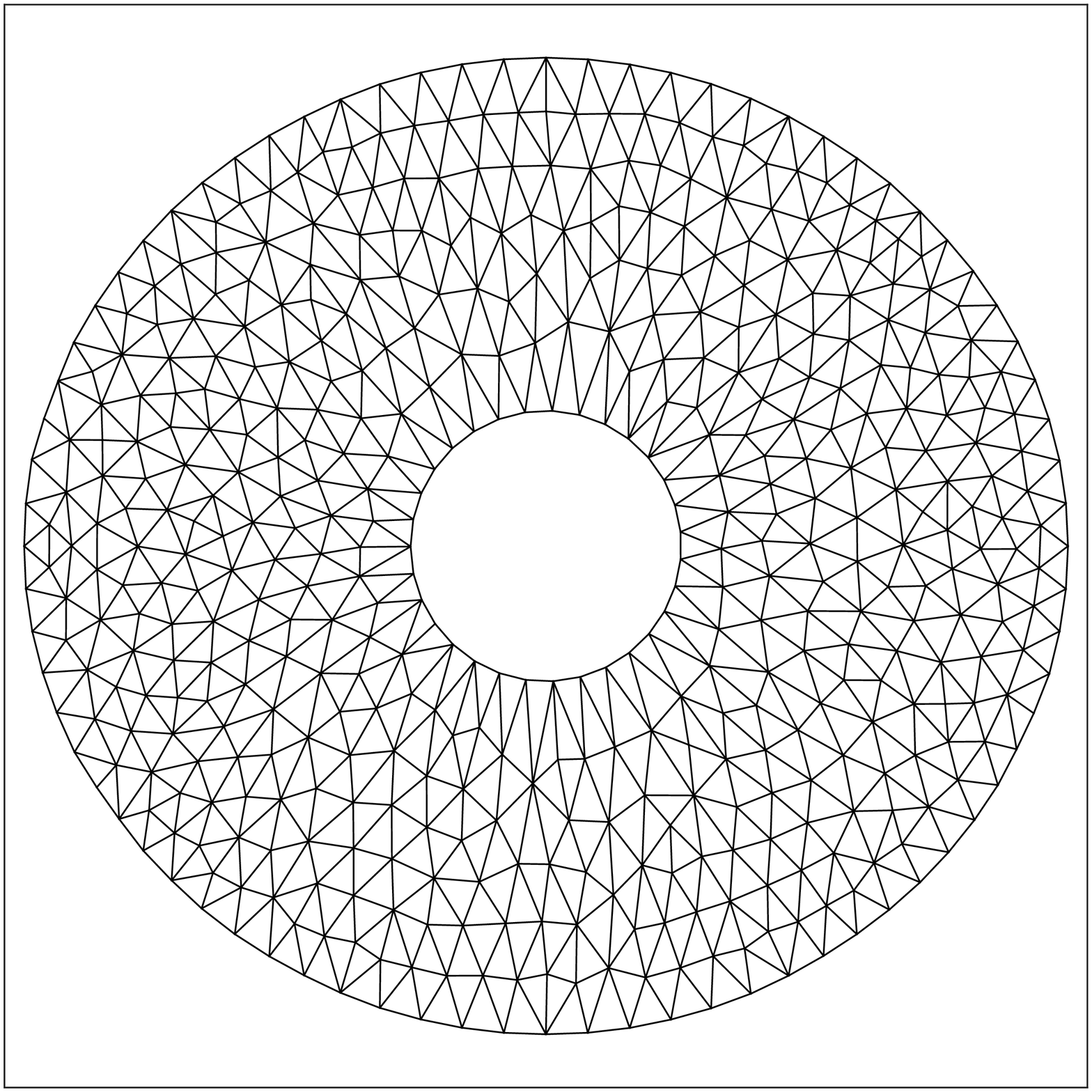}}
                 \caption{Mesh profile of $\Omega_h^{N_T}$ at $T=2$}
                \label{fig:Fig3b}
        \end{subfigure}%
\par\bigskip
        \begin{subfigure}[b]{0.49\textwidth}
                \centering
                \resizebox{0.9\textwidth}{!}{\includegraphics{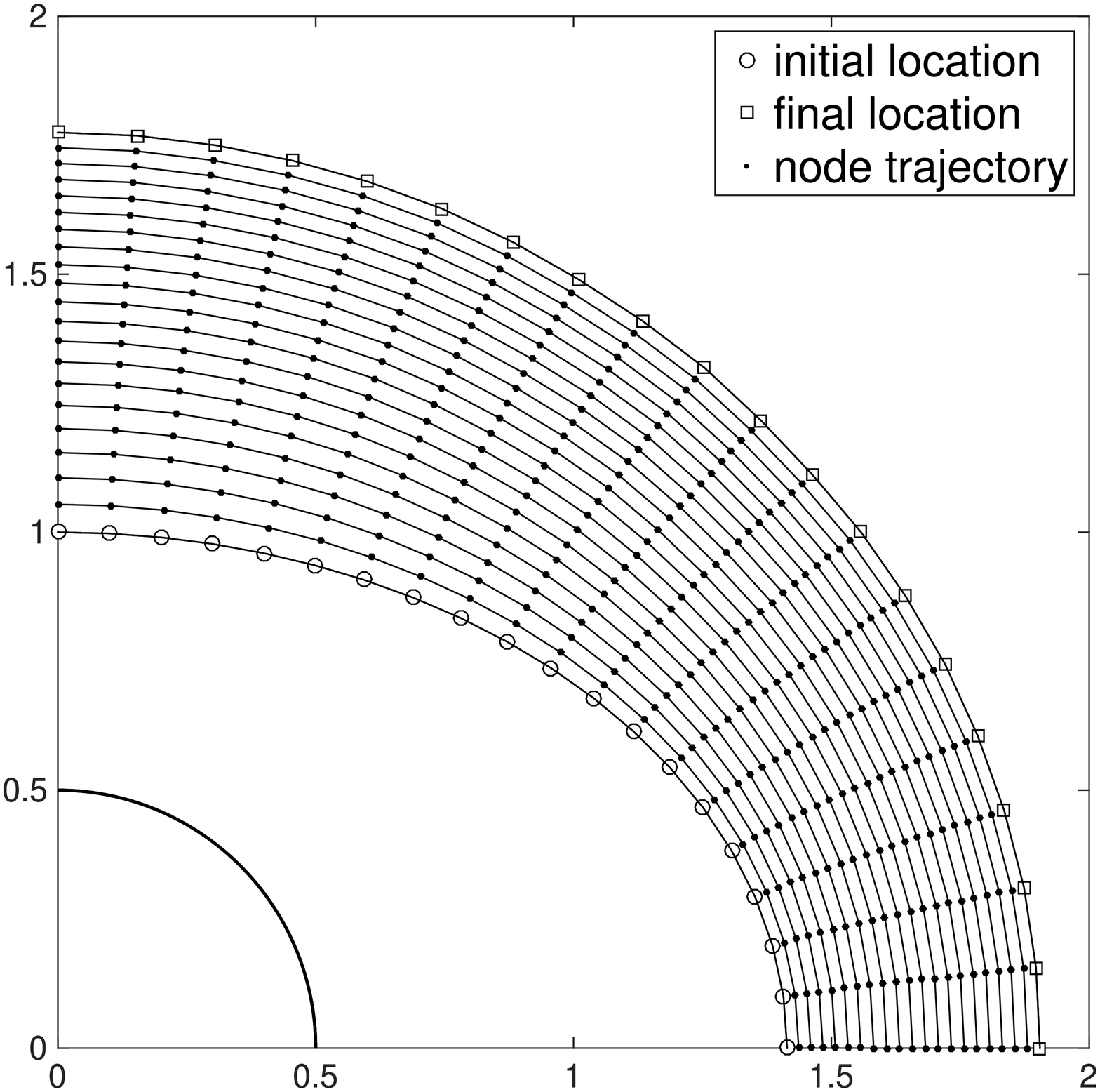}}
                \caption{Trajectory of boundary nodes }
                \label{fig:Fig3c}
        \end{subfigure}%
        \hfill
        \begin{subfigure}[b]{0.49\textwidth}
                \centering
                \resizebox{0.9\textwidth}{!}{\includegraphics{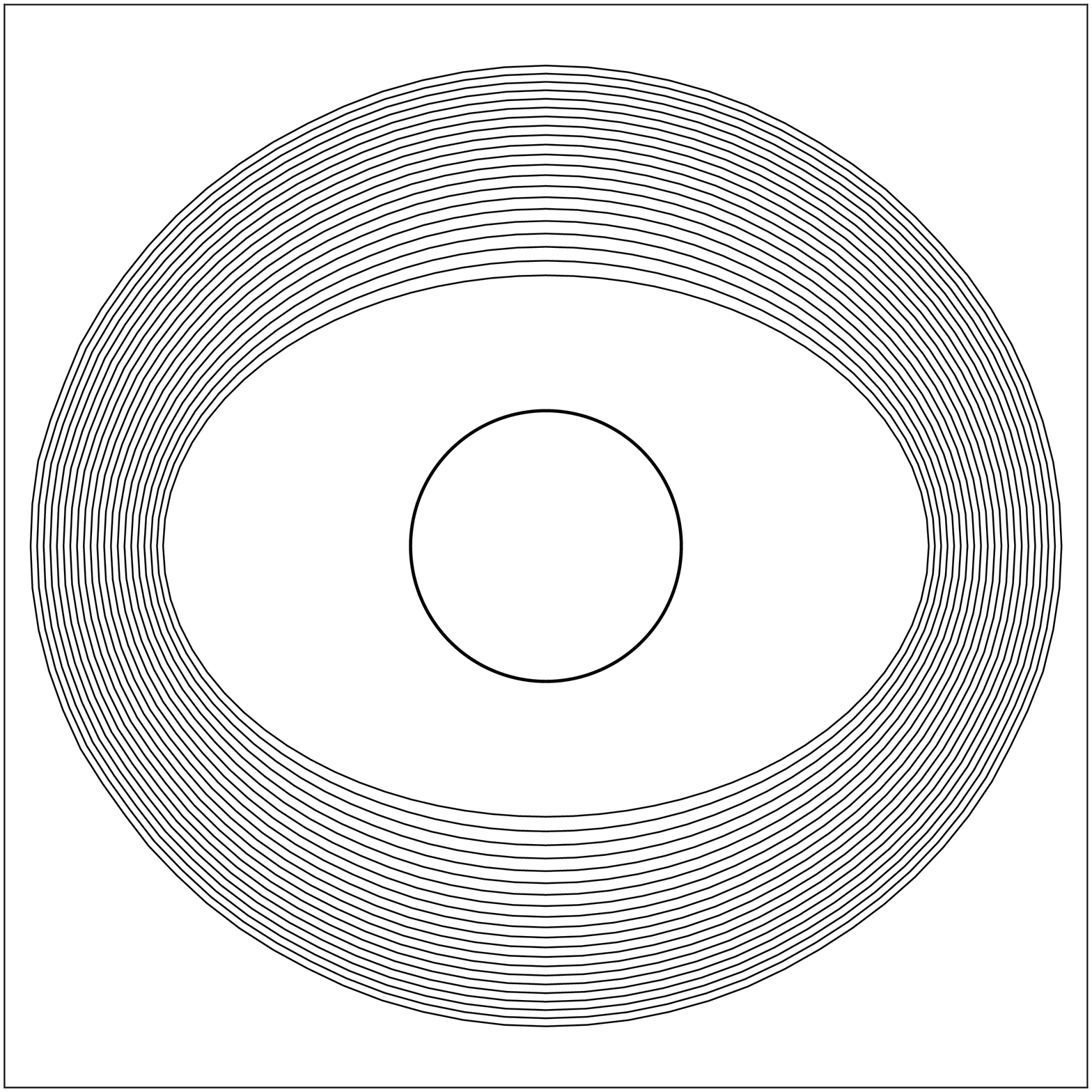}}
                \caption{Time evolution of the moving boundary}
                \label{fig:Fig3d}
        \end{subfigure}%
\caption{Computational results of Example \ref{example1}}
\label{fig:Fig3}
\end{figure}
The numerical results of the present experiment are shown in Fig. \ref{fig:Fig3}.
Fig. \ref{fig:Fig3b}, in particular, shows the shape of the annular domain at the final time of interest $T=2$ (with its mesh profile).
Meanwhile, Fig. \ref{fig:Fig3c} plots the trajectory of the boundary nodes, and we see from this figure that the nodes are well-spaced at every time-step.
The last plot, Fig. \ref{fig:Fig3d}, depicts the evolution of the (exterior) moving boundary $\Gamma^k_h$.
Here, the innermost exterior boundary represents the initial profile of $\Gamma^0_h$ and the outermost corresponds to its final shape.
As expected, the annular domain $\Omega(t)\setminus \overline{B}$ expands through time.
%
\subsection{EOC of CMM for the Hele-Shaw problem}
\label{sec:EOC-HS}

To check the accuracy of CMM for Hele-Shaw problems of the form \eqref{eq:general_Hele-Shaw}, with $\alpha=1$ and $\lambda=0$, we use the method of manufactured solutions \cite{SalariKnupp2000}.
Therefore, we construct a proper manufactured solution for the Hele-Shaw problem \eqref{eq:general_Hele-Shaw}.
\begin{proposition}\label{prop:Manufactured1}
  We suppose $\phi(x,t)$ is a smooth function with $\phi<0$ for $x \in \overline{B}$ and $|\nabla \phi| \neq 0$ on $\{\phi = 0\}$, for $t \in [0,T]$.
  We define $f:=\Delta \phi$, $q_B:= -\nabla \phi \cdot {\nu}$, $\bb{\gamma}:= \left(-\frac{\phi _t}{|\nabla \phi|^2}+1\right) \nabla \phi$ (where $\phi_t$ means the partial derivative of $\phi$ with respect to $t$), and $\Omega_0:=\{ \phi(x,0)<0 \}$.
  Then, $u(x,t)=-\phi(x,t)$ and $\Omega (t):= \{ x \in \mathbb{R}^d \mid \phi(x,t) < 0) \}$ satisfy the moving boundary value problem \eqref{eq:general_Hele-Shaw} with $\alpha = 1$ and $\lambda=0$.
\end{proposition}
\begin{proof}
The proposition is easily verified by straightforward computation noting that the normal velocity of the moving boundary $V_{n}$ and the unit normal vector ${\nu}$ with respect to the moving boundary $\Gamma(t)$ can be expressed in terms of the level set function $\phi$; that is, $V_{n} = -\frac{\phi_t}{|\nabla \phi|}$ and ${\nu} = \frac{\nabla \phi}{|\nabla \phi|}$ (see, e.g., \cite{KimuraNotes2008}).
\end{proof}
We check the experimental order of convergence (EOC) by comparing the approximate solution $u^k_h$ with the manufactured solution $\phi^k$.
In this case, $\phi^k$ is viewed as the interpolated exact solution to the solution space of the discretized problem.
Now, with regards to EOC, we define the numerical errors as follows:
\[
  \operatorname{err}_{\Gamma}  := \max_{0 \leqslant k \leqslant N_T} \max_{x \in \Gamma^k_h} \ \operatorname{dist}(x,\Gamma(k \tau)),\qquad
  \operatorname{err}_{\mathcal{X}^k} := \max_{0 \leqslant k \leqslant N_T} \left\{ \left\| u^k_h-\Pi_h u(\cdot,k\tau) \right\|_{\mathcal{X}^k} \right\},
\]
where $\mathcal{X}^k \in \{L^2(\Omega_h^k \setminus \overline{B}), H^1(\Omega_h^k \setminus \overline{B})\}$, and $\Pi_h :H^1(\Omega) \rightarrow P_1(\mathcal{T}_h(\Omega))$ is the projection map such that $\Pi_h u(p) = u(p)$ for all nodal points $p \in \mathcal{N}_h$ of $\mathcal{T}_h(\Omega)$.

\begin{numex}
\label{example2}
As an example, we perform a numerical experiment with the following conditions: $\varepsilon  \in \{ 10^{-4}, 10^{-2} \}$, $h \approx \tau = 0.05$,
\begin{align*}
	\phi(x,t)		&:=\frac{x_1^2}{2(t+1)}+\frac{x_2^2}{t+1}-1,  \quad t\in [0,1],  \quad (x:=(x_1,x_2)),
\end{align*}
\sloppy so $\overline{B} := \left\{x\in \mathbb{R}^2 \ \middle\vert \ x_1^2+x_2^2 \leqslant 0.5^2 \right\}$ and $\Omega_0 := \left\{x\in \mathbb{R}^2 \ \middle\vert \ 0.5 x_1^2+x_2^2 < 1 \right\}$.
\end{numex}
The computational results of Example \ref{example2} are shown in Fig. \ref{fig:Fig4}.
The initial profile of $\Gamma$ is depicted in Fig. \ref{fig:Fig4a}, while its shape at time $T=1$ is shown in Fig. \ref{fig:Fig4b}.
Meanwhile, Fig. \ref{fig:Fig4c} and Fig. \ref{fig:Fig4d} plot the evolution of the moving boundary $\Gamma$ (on the first quadrant) from initial to final time of interest $T=1$ with $\varepsilon = 10^{-4}$ and $\varepsilon = 10^{-2}$, respectively.
Notice that we get more stable evolution, in the sense that the boundary nodes are well-spaced at every time step, of the moving boundary for a higher value of $\varepsilon$ than with a lower value (refer, in particular, to the encircled region in the plots).
In fact, for higher values of $\varepsilon$, we observe better mesh quality than when $\varepsilon$ is of small magnitude.
Consequently, we notice in our experiment the obvious fact that there is a trade-off between accuracy and stability of the scheme when $\varepsilon$ is made smaller compared to the time step size $\tau \approx h$.
In fact, as already expected, the scheme is stable when the step size is taken relatively small compared to $\varepsilon$.
Results regarding accuracy are illustrated in further illustrations below.
\begin{figure}[htbp]
\centering
        \begin{subfigure}[b]{0.5\textwidth}
                \centering
                \resizebox{0.95\textwidth}{!}{\includegraphics{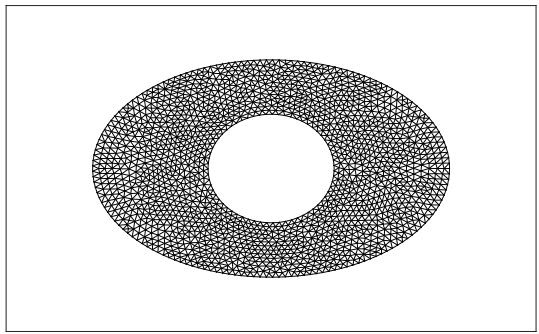}}
                \caption{Initial mesh profile $\Omega^{0}_h$ ($\varepsilon = 10^{-4}$)}
                \label{fig:Fig4a}
        \end{subfigure}%
        \hfill
        \begin{subfigure}[b]{0.5\textwidth}
                \centering
                \resizebox{0.95\textwidth}{!}{\includegraphics{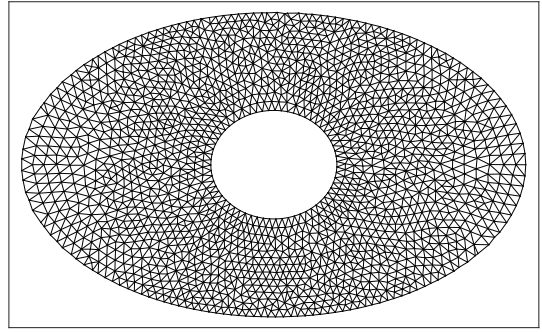}}
                \caption{Mesh profile of $\Omega^{N_T}_h$ at $T=1$ ($\varepsilon = 10^{-4}$)}
                \label{fig:Fig4b}
        \end{subfigure}%
        \par
        \bigskip
        \begin{subfigure}[b]{0.5\textwidth}
                \centering
                \resizebox{0.95\textwidth}{!}{\includegraphics{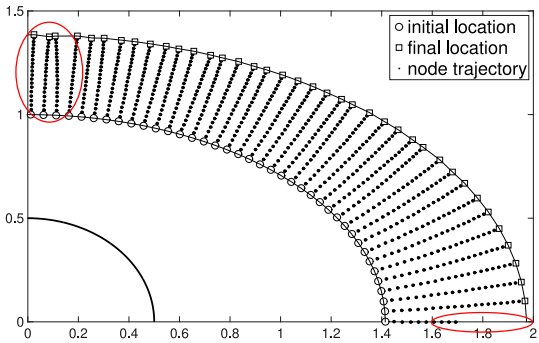}}
                \caption{Boundary nodes trajectory ($\varepsilon = 10^{-4}$)}
                \label{fig:Fig4c}
        \end{subfigure}%
        \hfill
                \begin{subfigure}[b]{0.5\textwidth}
                \centering
                \resizebox{0.95\textwidth}{!}{\includegraphics{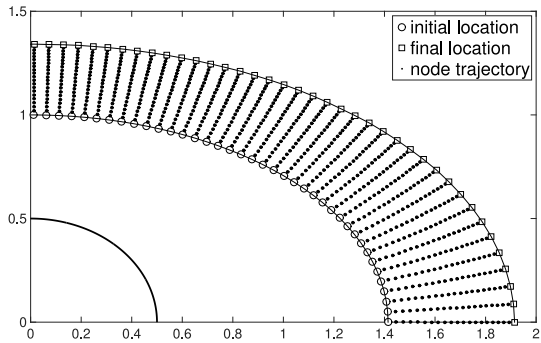}}
                \caption{Boundary nodes trajectory ($\varepsilon = 10^{-2}$)}
                \label{fig:Fig4d}
        \end{subfigure}%
\caption{Computational results for Example \ref{example2}}
\label{fig:Fig4}
\end{figure}

We also check how the error changes with the magnitude of the time step $\tau$ along with the maximum mesh size $h$ of the triangulation $\mathcal{T}_h$ by calculating the EOC of the present numerical example.
Here, the mesh size $h$ is as large as the time step $\tau$, i.e., $h \approx \tau$.
The results are depicted in Fig. \ref{fig:Fig5}.
Notice in these figures that the orders are mostly linear when $\varepsilon$ is sufficiently small, except, of course, in Fig. \ref{fig:Fig5b}.
Nevertheless, we can expect that the numerical solution converges to the exact solution by reducing the time step as well as the mesh size in the numerical procedure.
Based on these figures, the error is evidently reduced by choosing smaller $\varepsilon$.
However, in Fig. \ref{fig:Fig5b}, for sufficiently small $\varepsilon$, the errors become saturated and the saturated values decrease with order $O(\tau)$.
\begin{figure}[htbp]
        \begin{subfigure}[b]{0.49\textwidth}
                \centering
                \resizebox{\textwidth}{!}{\includegraphics{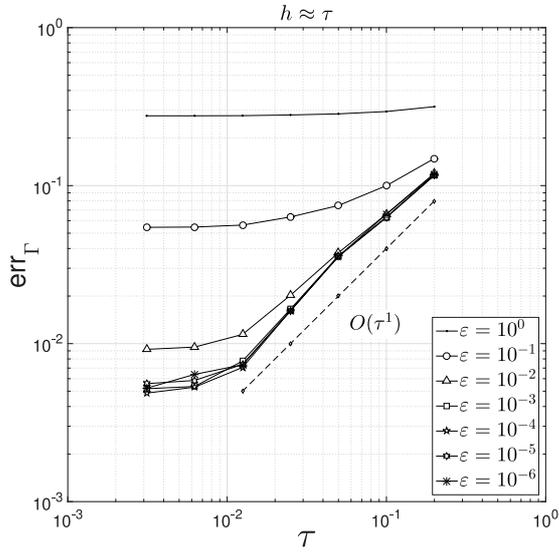}}
                \caption{$\tau$ vs $\operatorname{err}_{\Gamma}$ }
                \label{fig:Fig5a}
        \end{subfigure}%
        \hfill
        \begin{subfigure}[b]{0.49\textwidth}
                \centering
                \resizebox{\textwidth}{!}{\includegraphics{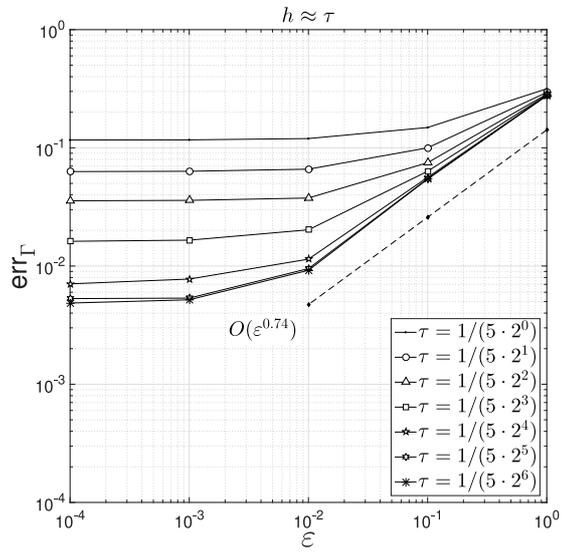}}
                \caption{$\varepsilon$ vs $\operatorname{err}_{\Gamma}$}
                \label{fig:Fig5b}
        \end{subfigure}%

	\par\bigskip
        \begin{subfigure}[b]{0.49\textwidth}
                \centering
                \resizebox{\textwidth}{!}{\includegraphics{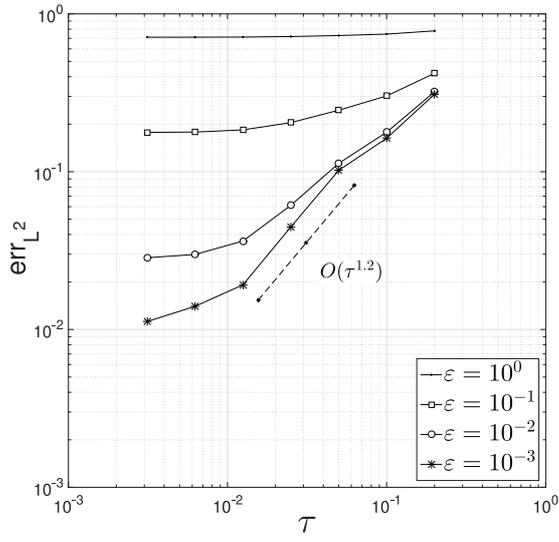}}
                \caption{$\tau$ vs $\operatorname{err}_{L^2}$ }
                \label{fig:Fig5c}
        \end{subfigure}%
        \hfill
        \begin{subfigure}[b]{0.49\textwidth}
                \centering
                \resizebox{\textwidth}{!}{\includegraphics{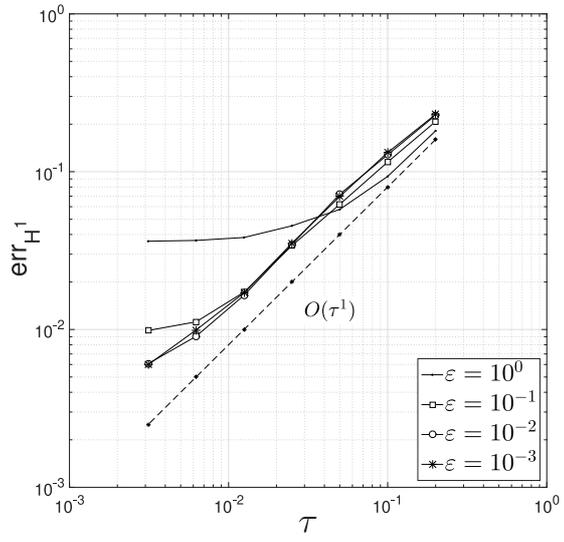}}
                \caption{$\tau$ vs $\operatorname{err}_{H^1}$}
                \label{fig:Fig5d}
        \end{subfigure}%
\caption{Error of convergences for Example \ref{example2}}
\label{fig:Fig5}
\end{figure}
\section{Bernoulli Free Boundary Problem}
\label{sec:Bernoulli}

In this section, we showcase the practicality of the method for solving stationary free boundary problems.

\subsection{Application of CMM in solving the Bernoulli free boundary problem}
\label{sec:BP}

	Here, we shall show that our proposed finite element scheme can actually be applied to numerically solved the well-known Bernoulli problem, a prototype of stationary free boundary problems \cite{FlucherRumpf1997}.
	The applicability of our method in solving the said problem is not surprising since the kinematic boundary condition $V_{n} = (- \nabla u + \bb{\gamma})\cdot {\nu} +\lambda$ with $\bb{\gamma} \equiv \bb{0}$ where $\lambda < 0$, in fact provides a descent direction for a gradient-based descent algorithm for solving the Bernoulli problem in the context of shape optimization.
	To see this, let us briefly discuss how the Bernoulli problem can be solved using the method of shape optimization, a well-established tool for solving FBPs \cite{DelfourZolesio2011}.

	The Bernoulli problem splits into two types: (i) the exterior case, similar to the topological profile of the domain $\Omega\setminus \overline{B}$ examined in previous sections, and (ii) the interior case which is the exact opposite of the exterior problem (i.e., the free boundary is the interior part of the disjoint boundaries).
	In the discussion that follows, we shall focus on the former case which can be described by the following overdetermined boundary value problem:
	\begin{equation}
	\label{Bernoulli_problem}
	\left\{\arraycolsep=1.4pt\def\arraystretch{1.}
	\begin{array}{rcll}
		-\Delta u		&=&0 		&\quad\text{in $\Omega\setminus \overline{B}$},\\
		u			&=&1		&\quad\text{on $\partial B$},\\
		u=0\quad\text{and}\quad\nabla u\cdot{\nu}&=&\lambda	&\quad\text{on $\Gamma$}.
	\end{array}
	\right.
	\end{equation}
	There are several ways to reformulate the above problem into a shape optimization setting (see, e.g., \cite{RabagoThesis2020} and the references therein), and the one we are concerned with here is the minimization of the shape functional \cite{EpplerHarbrecht2006}
	\[
		J(\Omega) = \int_{\Omega}\left( |\nabla (u(\Omega))|^2 + \lambda^2 \right) \operatorname{d}\!x,
	\]
	where $u = u(\Omega)$ is a unique weak solution to the underlying well-posed state problem:
  \vspace{2mm}

    Find $u \in H^1_{\Gamma,0}(\Omega \setminus \overline{B})$, with $u = 1$ on $\partial B$, such that
    \begin{equation}\label{eq:state}
     	\int_{\Omega \setminus \overline{B}} \nabla u : \nabla \varphi \ {\rm d}x = 0,
      		 \quad \forall \varphi \in H^1_{\Gamma,0}(\Omega \setminus \overline{B}).
  	\end{equation}
	We like to emphasize here that the positivity of the Dirichlet data on the fixed boundary $\partial B$ implies that the state solution $u$ is positive in $\Omega$.
	This, in turn, yields the identity $|\nabla u| \equiv - \nabla u \cdot \nu$ on $\Gamma$ because $u$ takes homogenous Dirichlet data on the free boundary $\Gamma$.

	The solution to the exterior Bernoulli problem \eqref{Bernoulli_problem} is equivalent to finding the solution pair $(\Omega, u(\Omega))$ to the shape optimization problem
	\begin{equation}
	\label{shape_problem}
		\min_{\Omega} J(\Omega),
	\end{equation}
	where $u(\Omega) \in H^1_{\Gamma,0}(\Omega \setminus \overline{B})$, with $u = 1$ on $\partial B$, satisfies the variational problem \eqref{eq:state}.
	This results from the necessary condition of a minimizer of the cost functional $J(\Omega)$, that is,
	\[
		\operatorname{d}\!J(\Omega)[\boldsymbol{V}]
			= \left. \frac{\operatorname{d}}{\operatorname{d}\varepsilon} J(\Omega_{\varepsilon})\right|_{\varepsilon = 0}
			= \int_{\Gamma} \left[ \lambda^2 - \left( \nabla u \cdot {\nu} \right)^2\right] V_{n} \operatorname{d}\!x
			=0,
			\quad V_{n}:= \boldsymbol{V} \cdot {\nu},
	\]
	has to hold for all sufficiently smooth perturbation fields $\boldsymbol{V}$.
	Here, $\Omega_\varepsilon$ stands for a deformation of $\Omega$ along the deformation field $\boldsymbol{V}$ vanishing on $\partial B$.
	For more details of how to compute $\operatorname{d}\!J(\Omega)[\boldsymbol{V}]$, and for more discussion on shape optimization methods, in general, we refer the readers to \cite{DelfourZolesio2011} and \cite{SokolowskiZolesio1992}.

	To numerically solve \eqref{shape_problem}, a typical approach is to utilize the \emph{shape gradient} (i.e., the kernel of the shape derivative $\operatorname{d}\!J(\Omega)[\boldsymbol{V}]$, see, e.g., \cite[Thm. 3.6, p. 479--480]{DelfourZolesio2011}) in a gradient-based descent algorithm.
	For instance, given enough regularity on the boundary $\Gamma$ and on the state $u$, we can take $\boldsymbol{0} \not\equiv \boldsymbol{V} = -\left[ \lambda^2 - \left(\nabla u \cdot {\nu}\right)^2\right]{\nu} \in L^2(\Gamma)$.
	This implies that, formally, for small $t>0$, we have the following inequality
	\begin{align*}
		J(\Omega_t)
			&= J(\Omega) + t \left. \frac{\operatorname{d}}{\operatorname{d}\varepsilon} J(\Omega_{\varepsilon})\right|_{\varepsilon = 0} + O(t^2)\\
			&= J(\Omega) + t \int_{\Gamma} \left[ \lambda^2 - \left(\nabla u \cdot {\nu} \right)^2\right] V_{n} \operatorname{d}\!s + O(t^2)\\
			&= J(\Omega) - t \int_{\Gamma} |V_{n}|^2 \operatorname{d}\!s + O(t^2)
			< J(\Omega).
	\end{align*}
	Here, we observe that we can simply take $(-\nabla u \cdot {\nu} + \lambda){\nu}$ as the descent vector $\boldsymbol{V}$.
	This issues from the fact that $\nabla u \cdot {\nu} + \lambda < 0$ on $\Gamma$ since $|\nabla u| \equiv - \nabla u \cdot \nu$ on $\Gamma$.
	Indeed, with $\boldsymbol{V} = (-\nabla u \cdot {\nu} + \lambda){\nu}$, we see that
	\begin{align*}
		J(\Omega_t)
			&= J(\Omega) + t \int_{\Gamma} \left(\nabla u \cdot {\nu} + \lambda \right) \left(-\nabla u \cdot {\nu} + \lambda\right) V_{n} \operatorname{d}\!s + O(t^2)\\
			&= J(\Omega) + t \int_{\Gamma} \underbrace{\left(\nabla u \cdot {\nu} + \lambda \right) }_{<\ 0} |V_{n}|^2 \operatorname{d}\!s + O(t^2)
			< J(\Omega).
	\end{align*}
	It is worth to mention here that simply taking the kernel of the shape derivative of the cost function (multiplied to the normal vector on the free boundary) as the deformation field $\boldsymbol{V}$ may lead to subsequent loss of regularity of the free boundary, hence forming oscillations of the free boundary.
	To avoid such phenomena, the descent vector is, in most cases, replaced by the so-called \emph{Sobolev gradient} \cite{Neuberger2010}.
	A strategy to do this is to apply the \emph{traction method} or the $H^1$ gradient method which are popular smoothing techniques in the field of shape design problems (see, e.g., \cite{AzegamiBook2020}).

	Now, the evolution of the free boundary $\Gamma(t)$ of the Bernoulli problem according to a shape gradient-based descent algorithm (see, e.g., \cite{EpplerHarbrecht2006}) describes a similar evolutionary equation for the Hele-Shaw problem with the moving boundary given as $\Gamma(t)$:
	\begin{equation}
	\label{Hele-Shaw-Bernoulli}
	\left\{\arraycolsep=1.4pt\def\arraystretch{1.1}
	\begin{array}{rcll}
		-\Delta u		&=&0 		&\quad\text{in $\Omega(t)\setminus \overline{B}$, \quad $t\in[0, T]$},\\
		u			&=&1		&\quad\text{on $\partial B$},\\
		u			&=&0		&\quad\text{on $\Gamma(t)$, \quad $t \in [0,T]$},\\
		V_{n} 			&=&- \nabla u\cdot{\nu} + \lambda	&\quad\text{on $\Gamma(t)$, \quad $t \in [0,T]$},\\
		\Omega(0)	&=& \Omega_0,
	\end{array}
	\right.
	\end{equation}
	where $T>0$.
Before we give a concrete numerical example illustrating the evolution of the solution of \eqref{Hele-Shaw-Bernoulli},
note that the convergence of the solution of the moving boundary problem in CMM to a stationary (non-degenerate) shape solution will be given in Section \ref{sec:properties} (see Proposition \ref{prop:convergence_to_a_stationary_point}).
Furthermore, we infer from this claim that the convergence of $\Omega(t)$ to $\Omega^\ast$, as time $t$ increases indefinitely, does not depend on the choice of the value of the parameter $\varepsilon$ in the $\varepsilon$-approximation of the normal-velocity flow $V_{n}\nu$ of the moving boundary $\Gamma(t)$.

\begin{numex}
\label{example3}
Let us now consider a concrete example of the exterior Bernoulli problem and apply CMM to approximate its numerical solution.
We consider the problem with $\lambda = -10$, and define the fixed interior domain as the \texttt{L}-shaped domain $B = (0.25,0.25)^2\setminus [0.25,0.25]$.
Also, we solve the problem for different choices of initial boundary $\Gamma(0)$.
In particular, we consider it to be a circle $\Gamma_1^0$, a square with rounded corners $\Gamma_2^0$, and a rectangle with rounded corners $\Gamma_3^0$.
We carry out the approximation procedure discretizing these domains with (initial) triangulations having mesh of width $h \approx 5 \times 2^3$.
We set the time step to $\tau = 0.001$ and take $T = 1$ as the final time of interest.
Hence, the procedure terminates after $N_T = 1000$ time steps.
Lastly, we set the CMM parameter $\varepsilon$ to $0.1$.
\end{numex}
	The results of the experiments are summarized in Fig. \ref{fig:Fig6}--Fig. \ref{fig:Fig9}.
	Fig. \ref{fig:Fig6} depicts the initial mesh triangulation of each mentioned test cases.
	Fig. \ref{fig:Fig7}, on the other hand, shows the mesh profile after $N_T$ time steps (i.e., the computational mesh profile at time $T=1$).
	Notice from these plots that the mesh quality actually deteriorates in the sense that the area of some of the triangles become very small (see the part of the discretized shape near the concave region of the domain in Fig. \ref{fig:Fig7}).
	This is not actually surprising since we do not imposed any kind of mesh improvement or re-meshing during the approximation process.
	Of course, as a consequence, the step size may become too large in comparison with the minimum mesh size of the triangulation after a large number of time steps have passed, and this may cause instability within the approximation scheme.
	Even so, we do not encounter this issue in these present test examples.
	Meanwhile, to illustrate how the nodes changes after each time step, we plot the boundary nodes' trajectory from initial to final time step for each test cases, and these are projected in Fig. \ref{fig:Fig8}.
	Moreover, in Fig. \ref{fig:Fig9a}, we plot the shapes at $T=1$ (i.e., the shape $\Gamma_i^k$, $i=1,2,3$, at final time step $k=N_T$) against each of the test cases.
	Notice that the computed shapes are slightly different, but are nevertheless close to the shape obtained via shape optimization methods (see \cite{RabagoThesis2020}).
	This is primarily due to the fact that the number of triangles within the initial mesh profile generated for each test cases are also different.
	Nonetheless, as we tested numerically, the resulting shapes at time $ T=1$ for each cases coincide at one another under smaller time steps and finer meshes.
	Lastly, Fig. \ref{fig:Fig9b} graphs the histories of the $L^2(\Gamma)$-norms between $\nabla u$ and $\lambda$, for each cases.

	We mention here that we also tested the case where the initial shape actually contains entirely the closure of the stationary shape (which is typically the experimental setup examined in the literature), and, as expected, we also get an almost identical shape with the ones obtained for the given cases.
	Here, we opted to consider the above-mentioned test setups to see whether our scheme works well in the case that the initial shape does not contain some regions of the stationary shape.
\begin{figure}[htbp]
\centering
        \begin{subfigure}[b]{0.32\textwidth}
                \centering
                \resizebox{\textwidth}{!}{\includegraphics{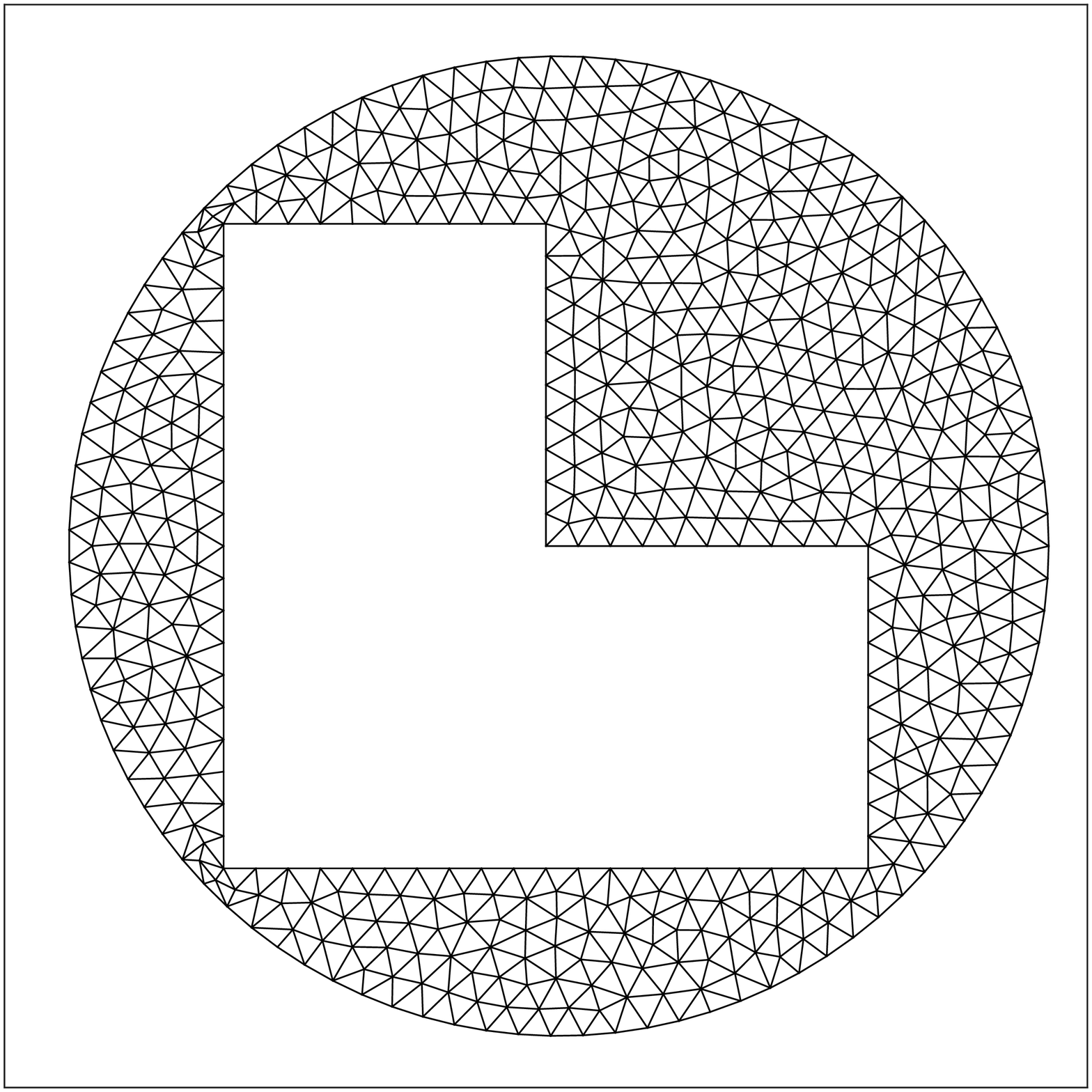}}
                \caption{Mesh profile of $\Gamma_1^0$}
                \label{fig:Fig6a}
        \end{subfigure}%
        \hfill
	 \begin{subfigure}[b]{0.32\textwidth}
                \centering
                \resizebox{\textwidth}{!}{\includegraphics{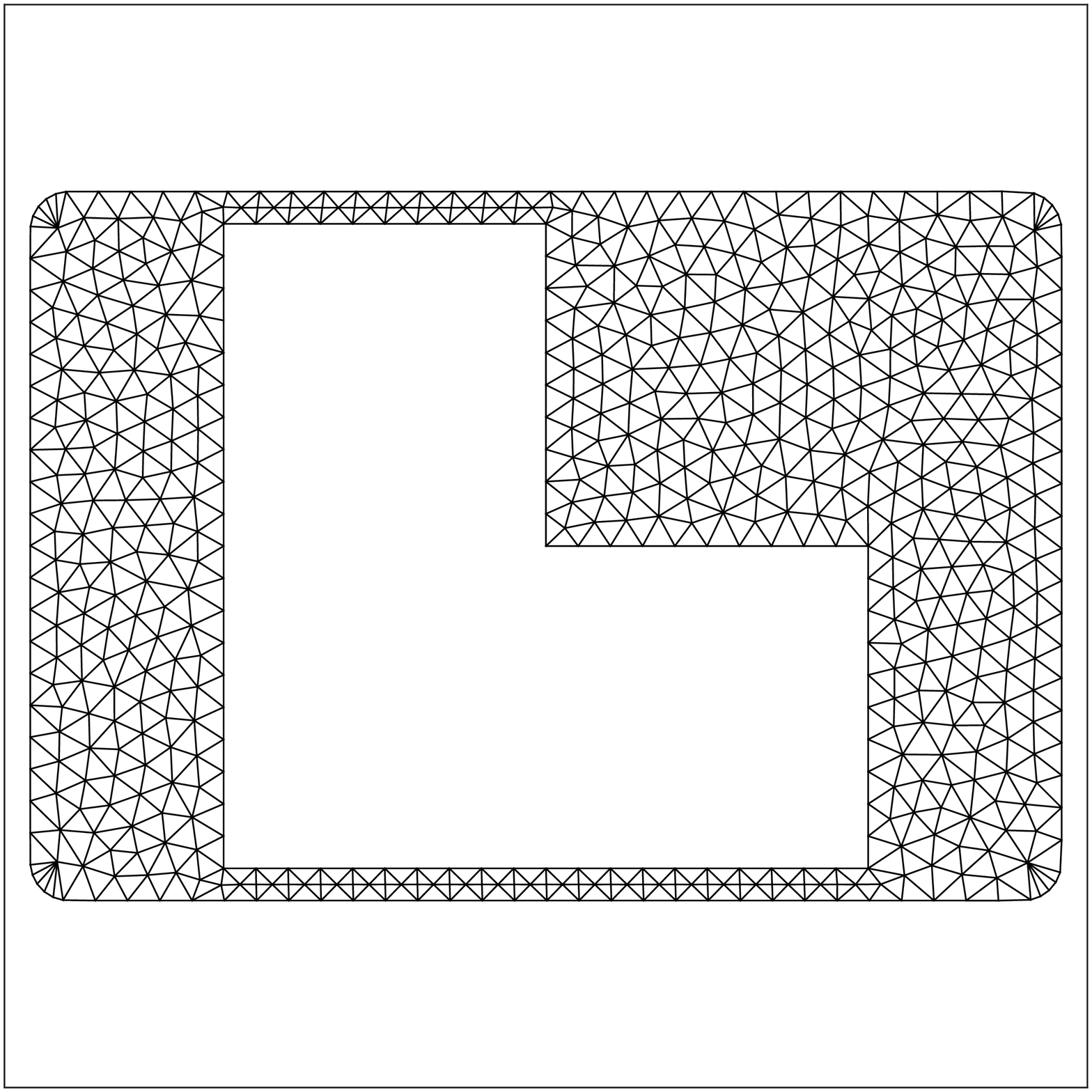}}
                 \caption{Mesh profile of $ \Gamma_2^0$}
                \label{fig:Fig6b}
        \end{subfigure}%
        \hfill
        \begin{subfigure}[b]{0.32\textwidth}
                \centering
                \resizebox{\textwidth}{!}{\includegraphics{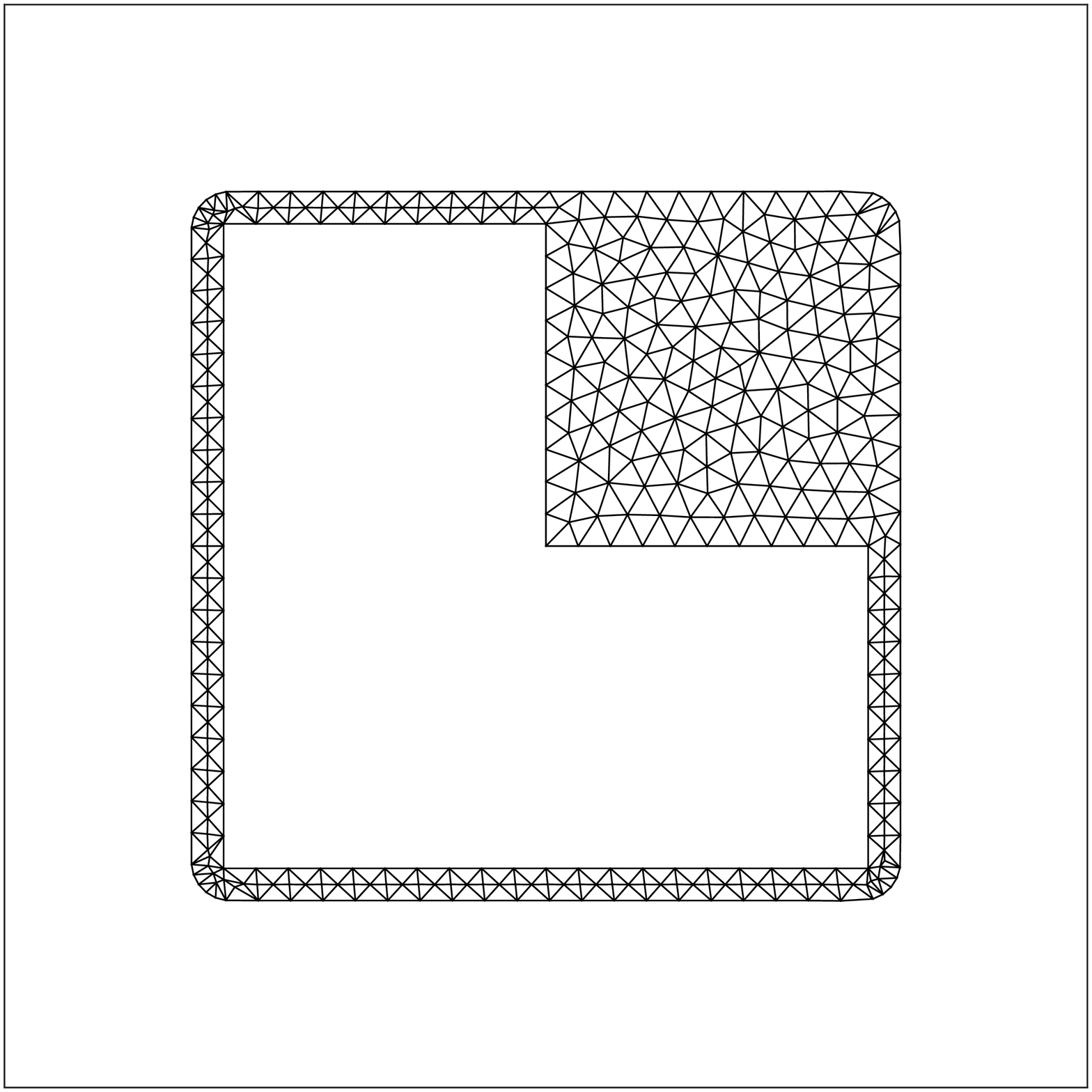}}
                \caption{Mesh profile of $\Gamma_3^0$}
                \label{fig:Fig6c}
        \end{subfigure}%
\caption{Initial computational meshes for each test case in Example \ref{example3}}
\label{fig:Fig6}
\end{figure}
\begin{figure}[htbp]
\centering
        \begin{subfigure}[b]{0.32\textwidth}
                \centering
                \resizebox{\textwidth}{!}{\includegraphics{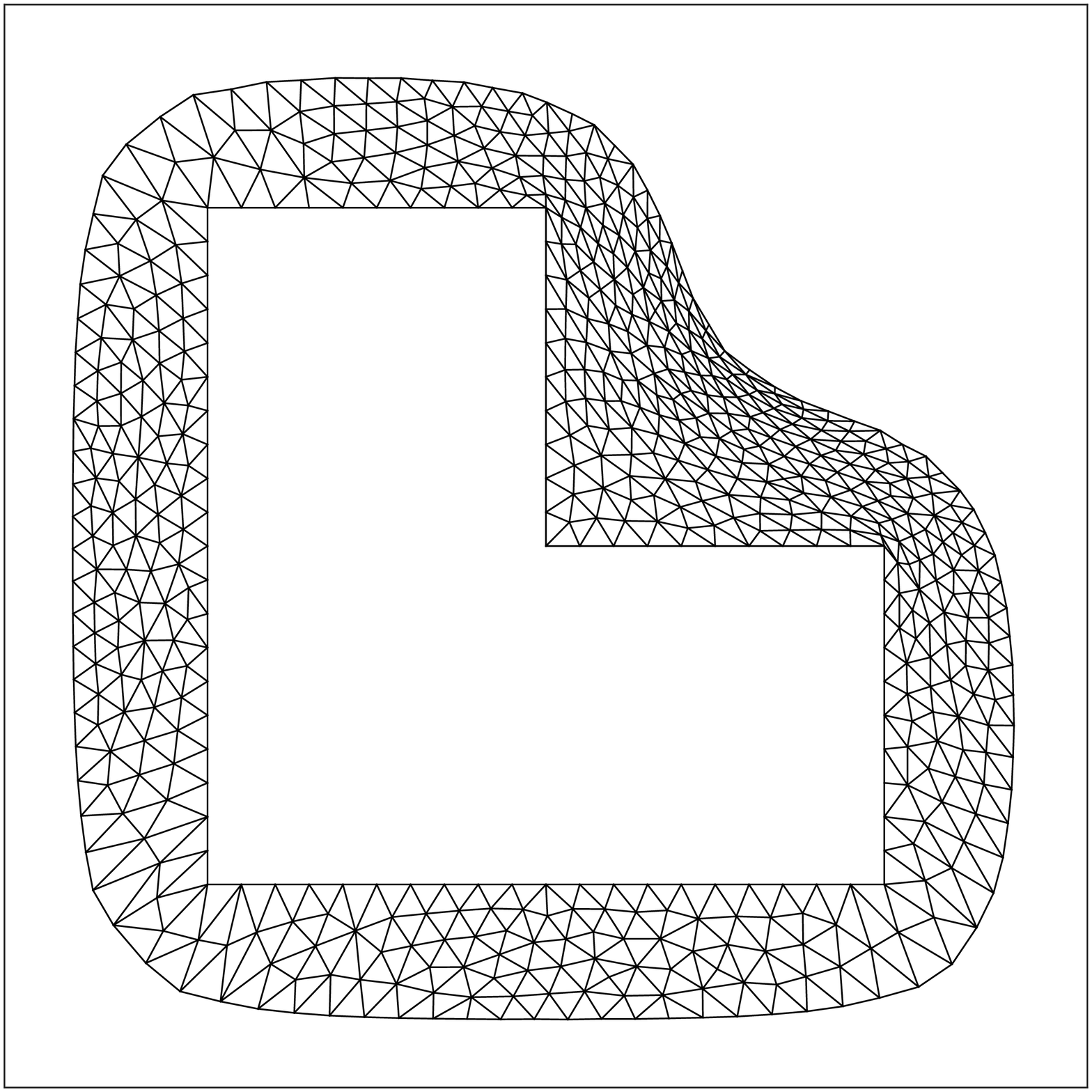}}
                \caption{Mesh profile of $\Gamma_1^{N_T}$}
                \label{fig:Fig7a}
        \end{subfigure}%
        \hfill
	 \begin{subfigure}[b]{0.32\textwidth}
                \centering
                \resizebox{\textwidth}{!}{\includegraphics{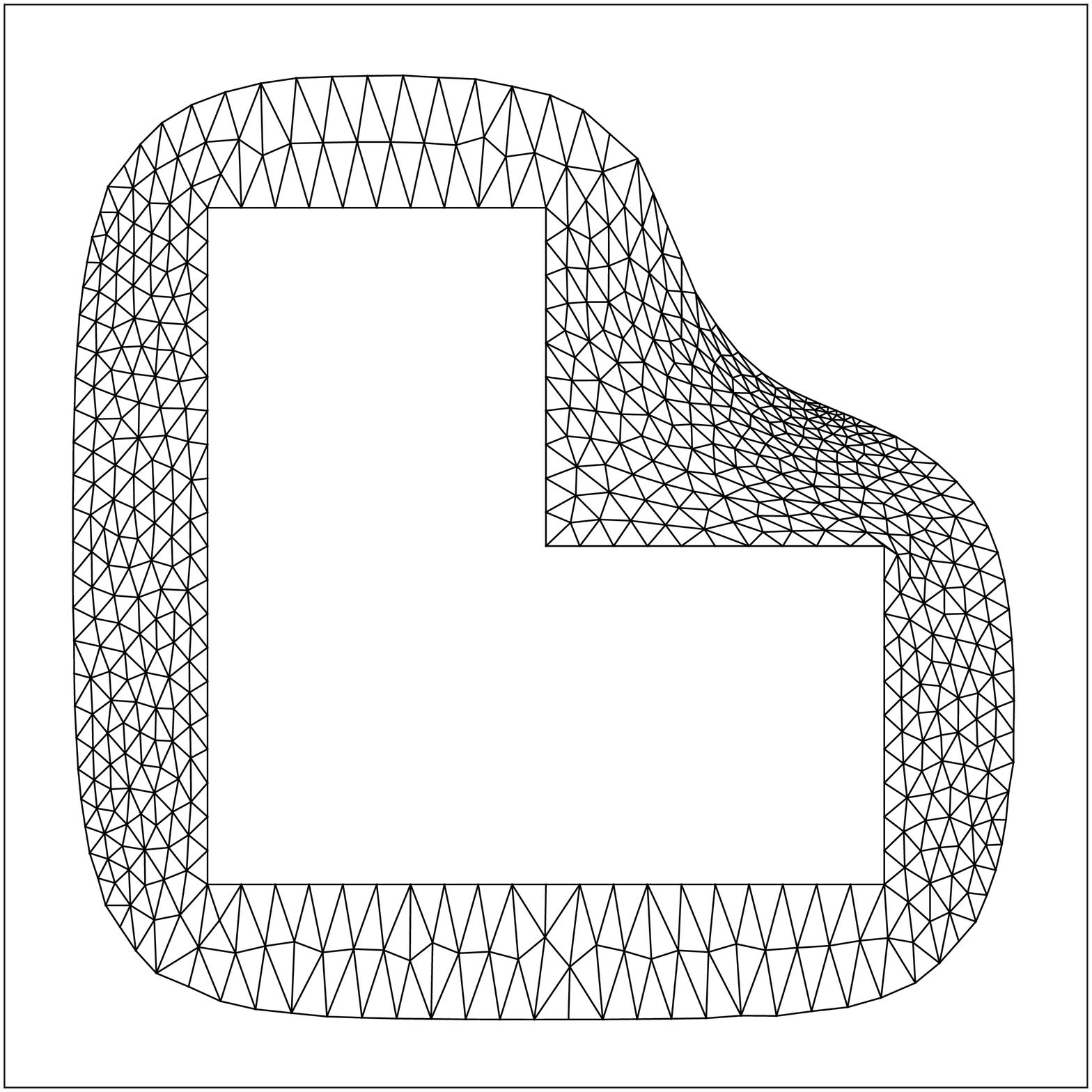}}
                 \caption{Mesh profile of $\Gamma_2^{N_T}$}
                \label{fig:Fig7b}
        \end{subfigure}%
        \hfill
        \begin{subfigure}[b]{0.32\textwidth}
                \centering
                \resizebox{\textwidth}{!}{\includegraphics{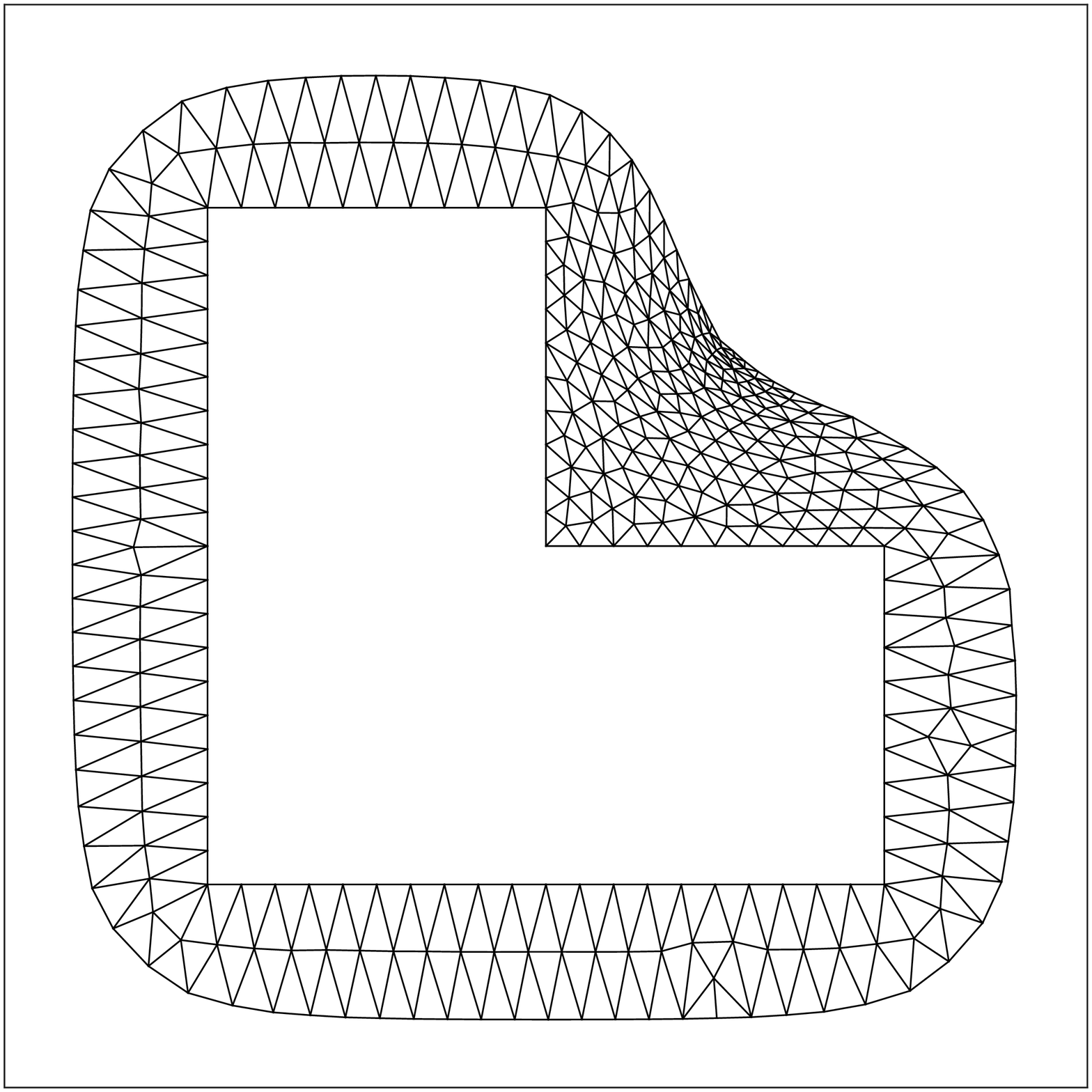}}
                \caption{Mesh profile of $\Gamma_3^{N_T}$}
                \label{fig:Fig7c}
        \end{subfigure}%
\caption{Computational mesh profiles for each test case in Example \ref{example3} at $T=1$}
\label{fig:Fig7}
\end{figure}
\begin{figure}[htbp]
\centering
        \begin{subfigure}[b]{0.32\textwidth}
                \centering
                \resizebox{\textwidth}{!}{\includegraphics{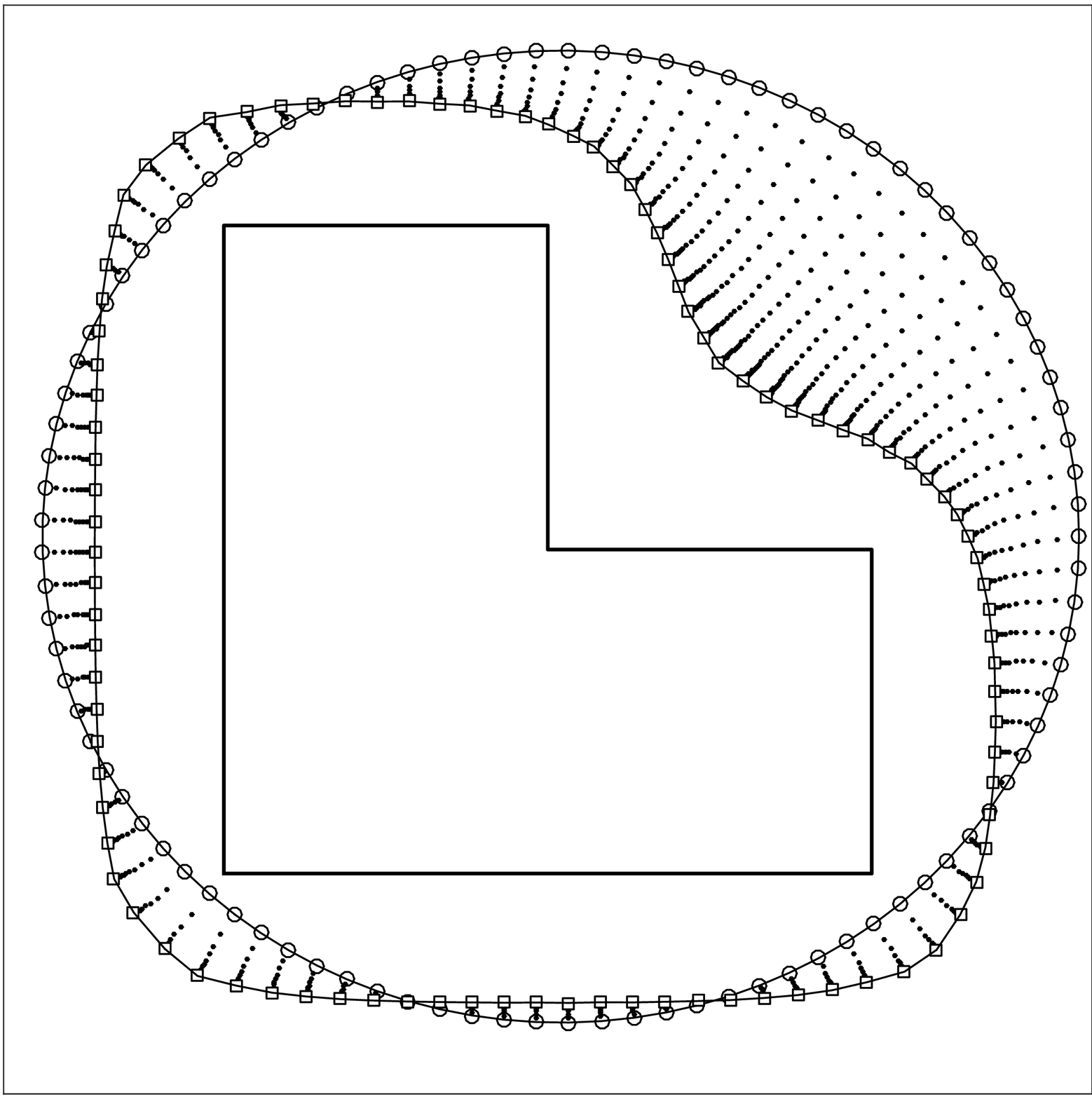}}
                \caption{Case $\Gamma(0)=\Gamma_1^0$}
                \label{fig:Fig8a}
        \end{subfigure}%
        \hfill
	 \begin{subfigure}[b]{0.32\textwidth}
                \centering
                \resizebox{\textwidth}{!}{\includegraphics{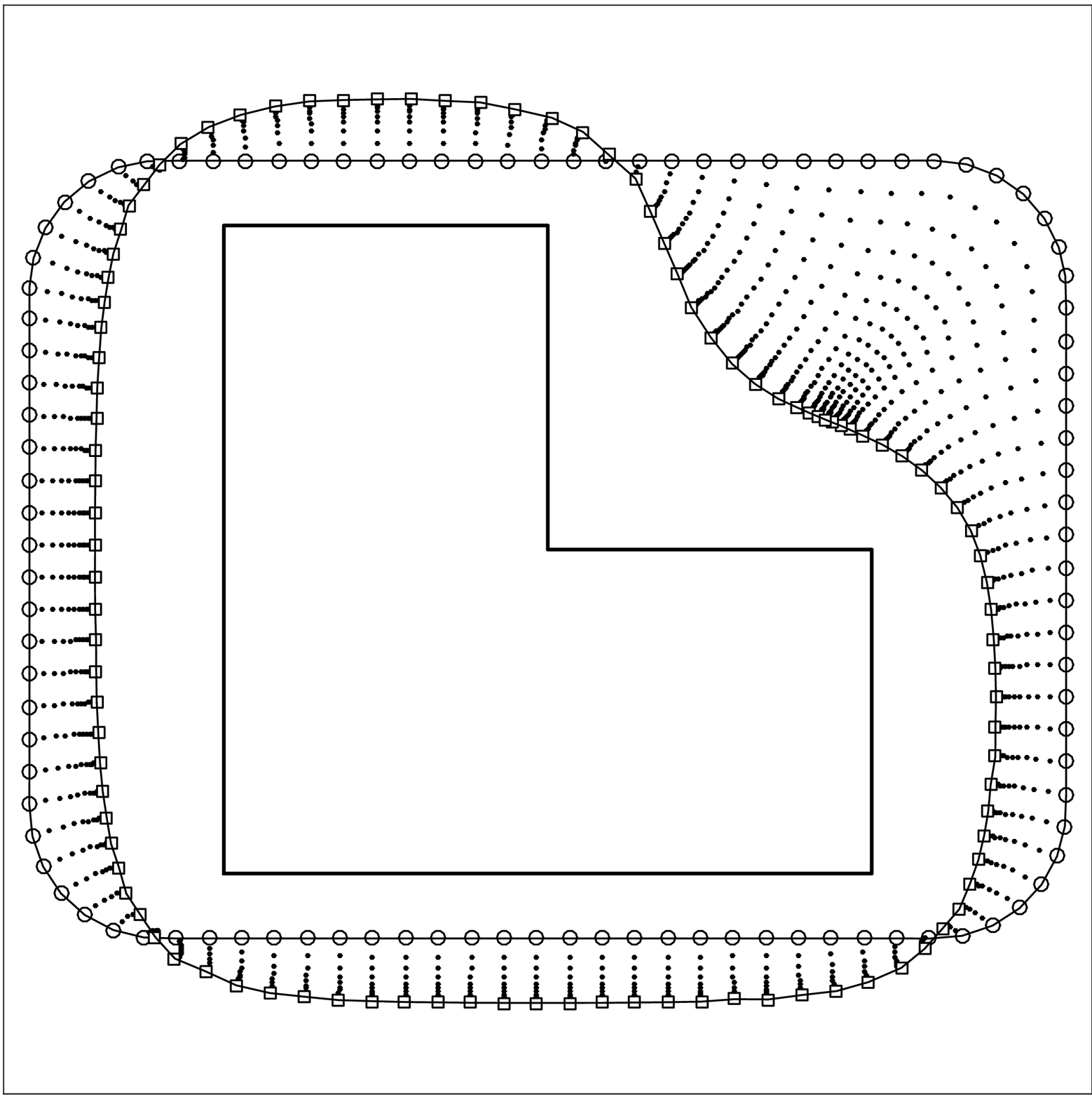}}
                 \caption{Case $\Gamma(0)=\Gamma_2^0$}
                \label{fig:Fig8b}
        \end{subfigure}%
        \hfill
        \begin{subfigure}[b]{0.32\textwidth}
                \centering
                \resizebox{\textwidth}{!}{\includegraphics{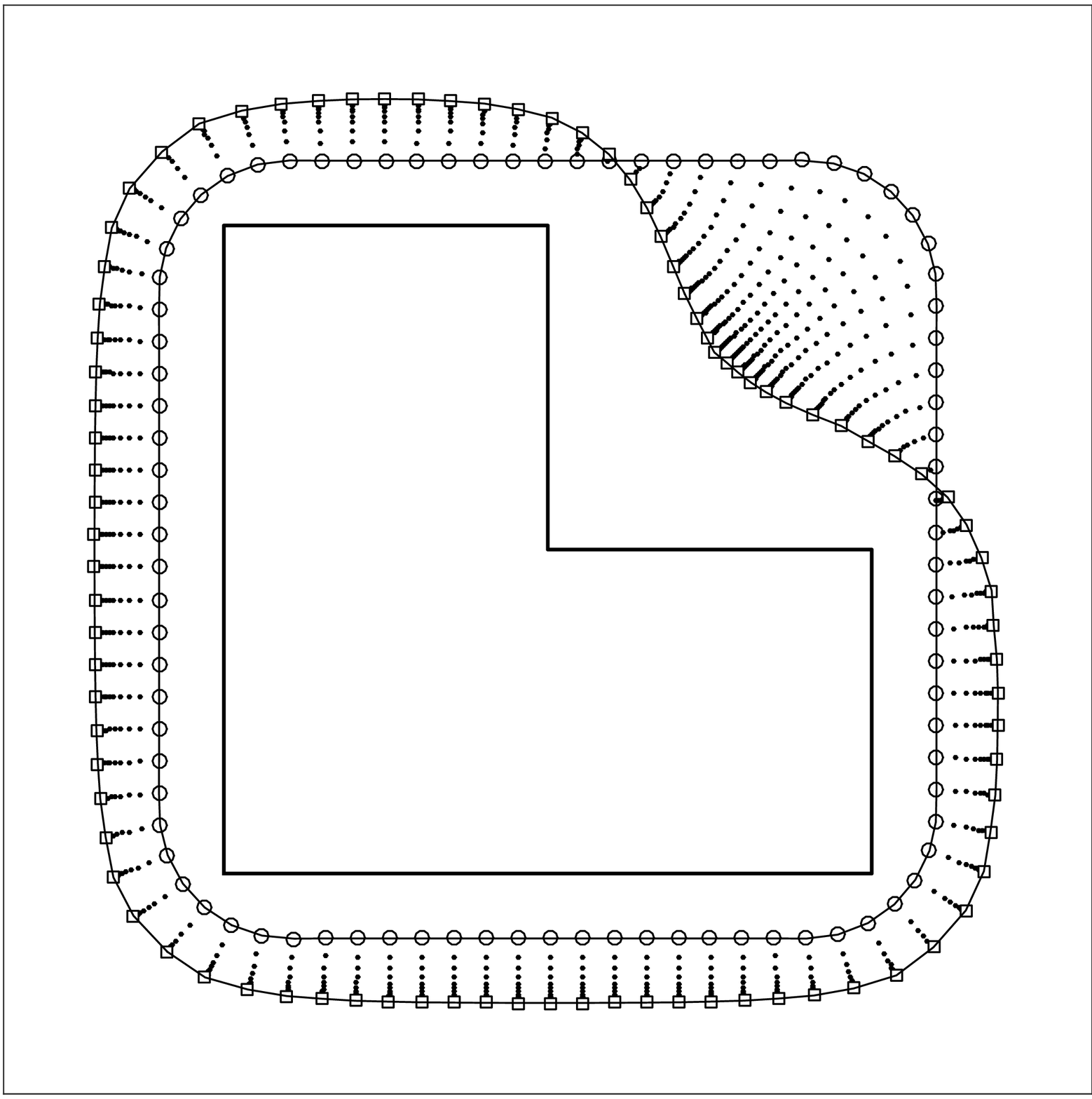}}
                \caption{Case $\Gamma(0)=\Gamma_3^0$}
                \label{fig:Fig8c}
        \end{subfigure}%
\caption{Boundary nodes' trajectories for each test case in Example \ref{example3}}
\label{fig:Fig8}
\end{figure}
\begin{figure}[htbp]
\centering
        \begin{subfigure}[b]{0.5\textwidth}
                \centering
                \resizebox{0.9\textwidth}{!}{\includegraphics{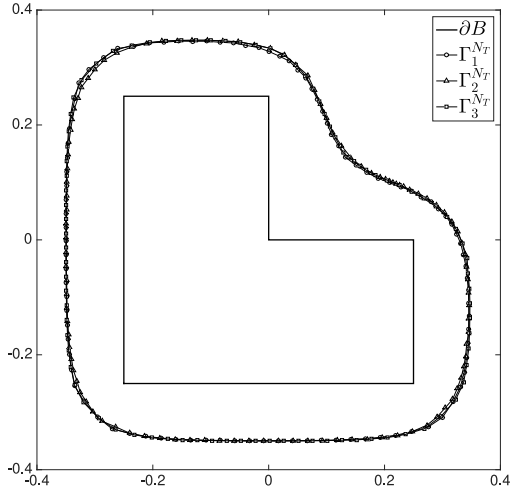}}
                \caption{Computed shapes}
                \label{fig:Fig9a}
        \end{subfigure}%
        \hfill
	 \begin{subfigure}[b]{0.5\textwidth}
                \centering
                \resizebox{0.9\textwidth}{!}{\includegraphics{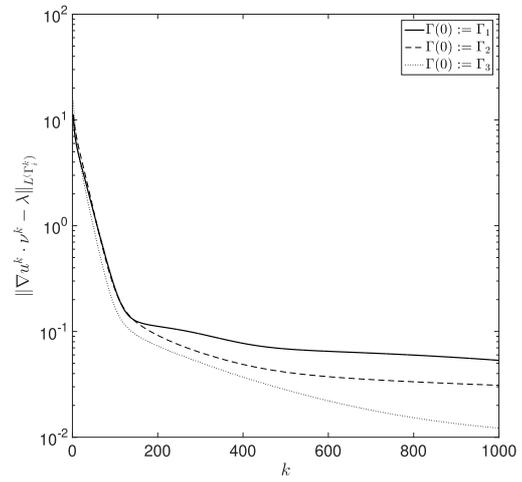}}
                 \caption{Histories of values of $\|\nabla u^k \cdot \nu^k  - \lambda\|_{L^2(\Gamma_i^k)}$}
                \label{fig:Fig9b}
        \end{subfigure}%
\caption{Plot \ref{fig:Fig9a}: Cross comparison of computed shapes at $T=1$; plot \ref{fig:Fig9b}: history of the $L^2$-norm $\|\nabla u^k \cdot \nu^k  - \lambda\|_{L^2(\Gamma_i^k)}$, $i=1,2,3$, for Example \ref{example3}}
\label{fig:Fig9}
\end{figure}
\section{Mean Curvature Flow Problem}
\label{sec:MCF}
\subsection{Application of CMM to mean curvature flow problem}

As further application of CMM, we will showcase in this section how CMM can easily be adapted to handle mean curvature flows:
\begin{equation}
  \label{eq:curvature flow}
    V_{n} = -\kappa \qquad \text{on $\Gamma(t)$},
\end{equation}
where, $\kappa$ denotes curvature of $\Gamma(t)$ for $d=2$, or the sum of principal curvature of $\Gamma(t)$ for $d \geqslant 3$.
The corresponding problem under this situation is often referred to in the literature as the \emph{curve shortening problem} when $d=2$ (see, e.g., \cite{GageHamilton1986,Grayson1987}), and is called, in general (i.e., $d \geqslant 3$), as the \emph{mean curvature flow problem} (see, e.g., \cite{Dziuk1991,Huisken1984}).
Here, we use the latter terminology in any dimensional case.
For other numerical methods used to solve the problem such as the CSM coupled with the level-set method, or via a finite element method using approximation by a reaction-diffusion equation, we refer the readers to \cite{KimuraNotsu2002} and \cite{NochettoVerdi1996}, respectively.

Now, let $\kappa^k$ be the curvature of $\Gamma^k = \partial \Omega ^k$. Similarly to \eqref{eq:velocity}, the smooth extension of $V_{n}{\nu}$ according to CMM satisfies the following problem for ${\bb{w}^k_h} : \Omega_h^k\setminus \overline{B_h} \to \mathbb{R}^d$:

\begin{equation}
\label{eq:curv-velocity}
\left\{\arraycolsep=1.4pt\def\arraystretch{1}
\begin{array}{rcll}
  - \Delta {\bb{w}^k} 		&=& \bb{0}	&\quad \text{in $\Omega_h^k \setminus \overline{B_h}$},\\[0.3em]
  {\bb{w}^k} 			&=& \bb{0}	&\quad \text{on $\partial B_h$},\\[0.3em]
  {\varepsilon}  \nabla \bb{w}^k \cdot \nu^k  + {\bb{w}^k} &=& -\kappa^k {\nu}^k		&\quad \text{on $\Gamma_h^k$}.
\end{array}
\right.
\end{equation}

In variational form, the system of partial differential equations \eqref{eq:curv-velocity} is given as follows:
find ${\bb{w}^k} \in H^1_{\partial B,\bb{0}}(\Omega^k\setminus \overline{B};\mathbb{R}^d)$ such that
	\begin{align}
   	&\displaystyle \int_{\Omega^k \setminus \overline{B}} \nabla {\bb{w}^k} : \nabla \bb{\varphi} \ {\rm d}x
    		 + \frac{1}{{\varepsilon}} \int_{\Gamma^k}  {\bb{w}^k} \cdot \bb{\varphi}\ {\rm d}s \nonumber\\
			&\displaystyle \hspace{1in} = -\frac{1}{{\varepsilon}} \int_{\Gamma^k} \kappa^k {\nu}^k \cdot \bb{\varphi}\ {\rm d}s\nonumber\\
      &\displaystyle \hspace{1in} = -\frac{1}{{\varepsilon}} \int_{\Gamma^k} \operatorname{div}_{\Gamma} \bb{\varphi} \ {\rm d}s,
    				\quad \forall \bb{\varphi} \in H_{\partial  B,\bb{0}}^1(\Omega^k\setminus \overline{B};\mathbb{R}^d),\label{curvatureflow_weakform}
	\end{align}
   where $\operatorname{div}_{\Gamma}$ denotes the \emph{tangential divergence} (see, e.g., \cite[Chap. 9, Sec. 5.2, eq. (5.6), p. 495]{DelfourZolesio2011} or \cite[Chap. 3, Sec. 1, Def. 2.3, p. 53]{KimuraNotes2008}).
Evaluating the mean curvature term numerically is quite problematic, especially when implemented in a finite element method.
Here, however, we point out that to numerically evaluate the integral consisting of the mean curvature $\kappa$, one may utilize the so-called \emph{Gauss-Green formula} on $\Gamma$ (see, e.g., \cite[Chap. 2, Sec. 2, Thm. 2.18, p. 56]{KimuraNotes2008} or \cite[eq. (5.27), p. 498]{DelfourZolesio2011}):
\begin{equation}
	\label{eq:Gauss_Green_formula}
	\int_{{\Gamma}}{\kappa {\nu} \cdot \bb{v}}\ {\rm d}s
		= \int_{{\Gamma}}{ \operatorname{div}_{\Gamma} \bb{v} } \ {\rm d}s,
\end{equation}
which is valid for $C^2$ regular boundary $\Gamma$ and vector-valued function $ \bb{v} : \Gamma \to \mathbb{R}^d$ that belongs at least to $C^1(\Gamma; \mathbb{R}^d)$.
Hence, the variational problem \eqref{curvatureflow_weakform} can be solved at once without the need to evaluate the mean curvature $\kappa^k$ at every time step $k = 0,1,\cdots,N_T$.

	To implement in a finite element method the right side integral appearing in the variational problem \eqref{curvatureflow_weakform}, we remark that the identity
	$\operatorname{div}_{\Gamma} \bb{\varphi} = \operatorname{div} \bb{\varphi} -(\bb{\varphi} \cdot \nu) \cdot\nu$ on $\Gamma$, actually holds for smooth $\Gamma$ and $\bb{\varphi} : \overline{\Omega} \to \mathbb{R}^d$.
	So, for a polygonal mesh $\Omega_h$ and $\Gamma_h:= \partial \Omega_h$, with triangular mesh $\mathcal{T}_h$ and element $\varphi_{ih} \in P_l(\mathcal{T}_h)$ ($i=1, 2, \ldots, d$), $l \in \mathbb{N}$, we have
	\[
	 \int_{\Gamma_h} \operatorname{div}_{\Gamma_h} \bb{\varphi}_h \ {\rm d}s
	 	=  \int_{\Gamma_h} \left( \operatorname{div} \bb{\varphi}_h - \frac{\partial \bb{\varphi}_h}{\partial \nu} \cdot\nu \right) {\rm d}s	.
	\]

\begin{numex}
\label{example4}
With the above at hand, we perform a numerical experiment for the mean curvature flow problem which we execute under the following conditions: ${\varepsilon}  = 0.1$, $\tau = 5 \cdot 10^{-4}$, with the maximum mesh size of width $h \approx 0.2$, $t \in [0,T]$, $T=1$,
$\Omega_0 := \left\{ (r,\theta) \in \mathbb{R}^2 \middle\vert\ 0 \leqslant r < \frac{2}{2-\cos(5\theta)},\ 0 \leqslant \theta \leqslant 2\pi \right\}$,
and $\overline{B}$ is the circle $C(\bb{0},0.5)$ as in Example \ref{example2}.
\end{numex}
The results of the experiment are summarized in Fig. \ref{fig:Fig10}.
Here, the initial, plotted with its mesh triangulation, is shown in Fig. \ref{fig:Fig10a}.
Fig. \ref{fig:Fig10b}, on the other hand, plots the mesh profile at selected time steps.
The third figure, Fig. \ref{fig:Fig10c}, depicts the evolution of the moving boundary from its initial profile (outermost exterior boundary) up to its final shape (innermost exterior boundary), and at some intermediate time steps.
Fig. \ref{fig:Fig10d} again plots the time evolution of the moving boundary, but now viewed on the first quadrant and with emphasis to the location of the boundary nodes at time steps $k = 100j$, for $j=0, 1, \ldots, 20$.
As expected, the curvature flow equation $V_{n} = -\kappa$ on $\Gamma(t)$ has the effect of flattening uneven parts of the boundary, hence shrinking the whole domain into the geometric profile of the interior boundary $\partial B$ (as evident in the figures), after a sufficiently large time has passed.
\begin{figure}[htbp]
\centering
        \begin{subfigure}[b]{0.5\textwidth}
                \centering
                \resizebox{0.9\textwidth}{!}{\includegraphics{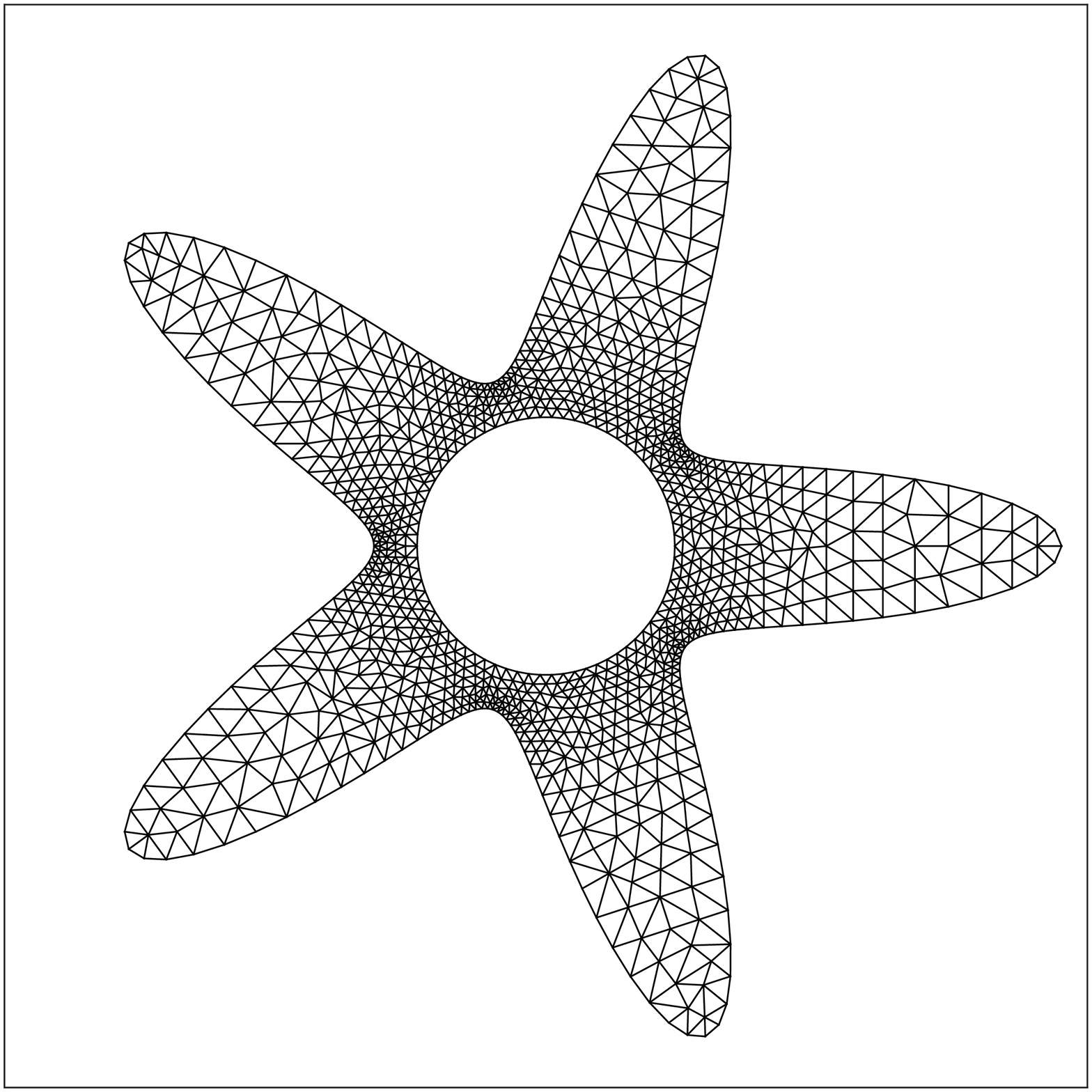}}
                \caption{Initial mesh profile of $\Omega_h^0$}
                \label{fig:Fig10a}
        \end{subfigure}%
        \hfill
	 \begin{subfigure}[b]{0.5\textwidth}
                \centering
                \resizebox{0.445\textwidth}{!}{\includegraphics{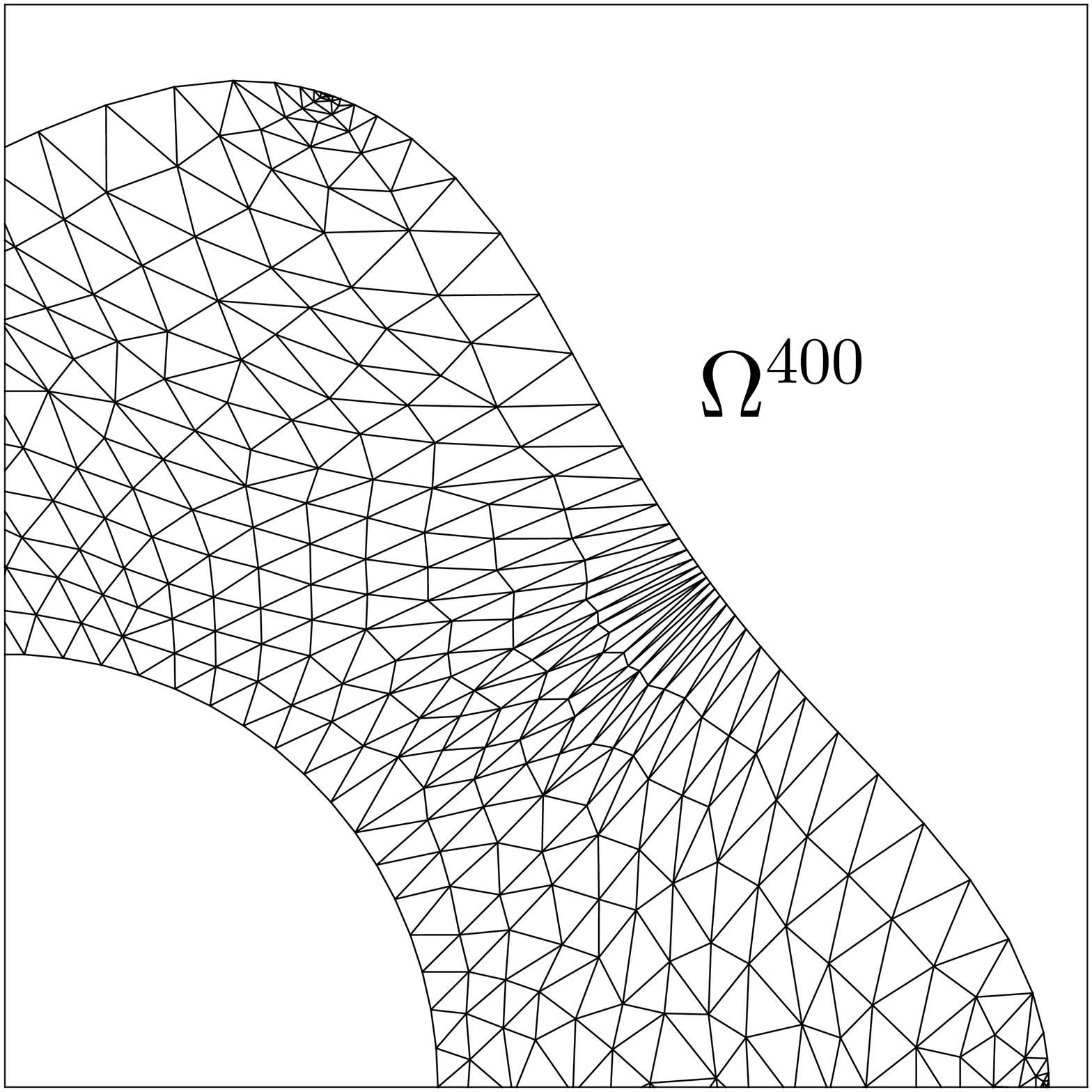}}
                \resizebox{0.445\textwidth}{!}{\includegraphics{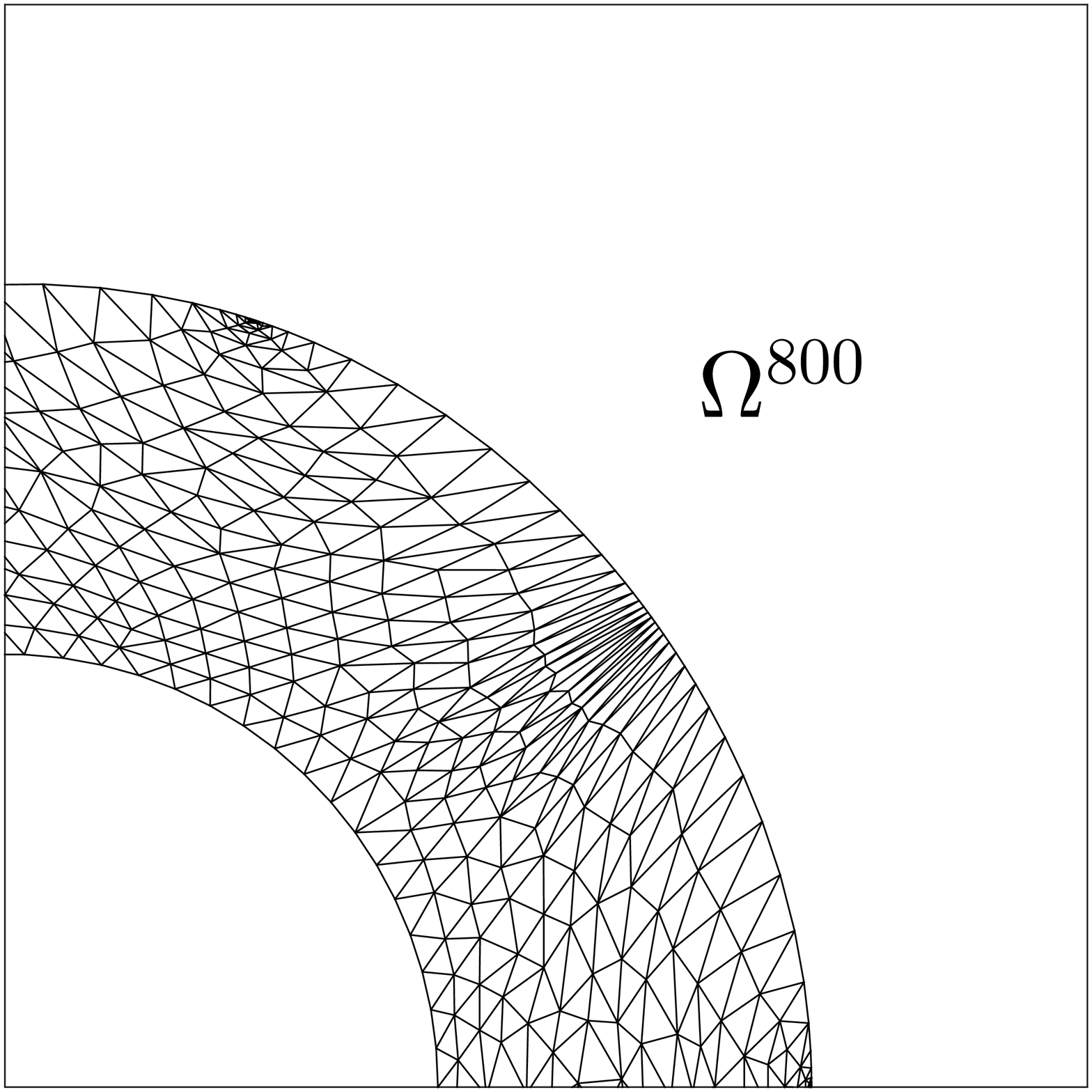}}
                \vskip1pt
                \resizebox{0.445\textwidth}{!}{\includegraphics{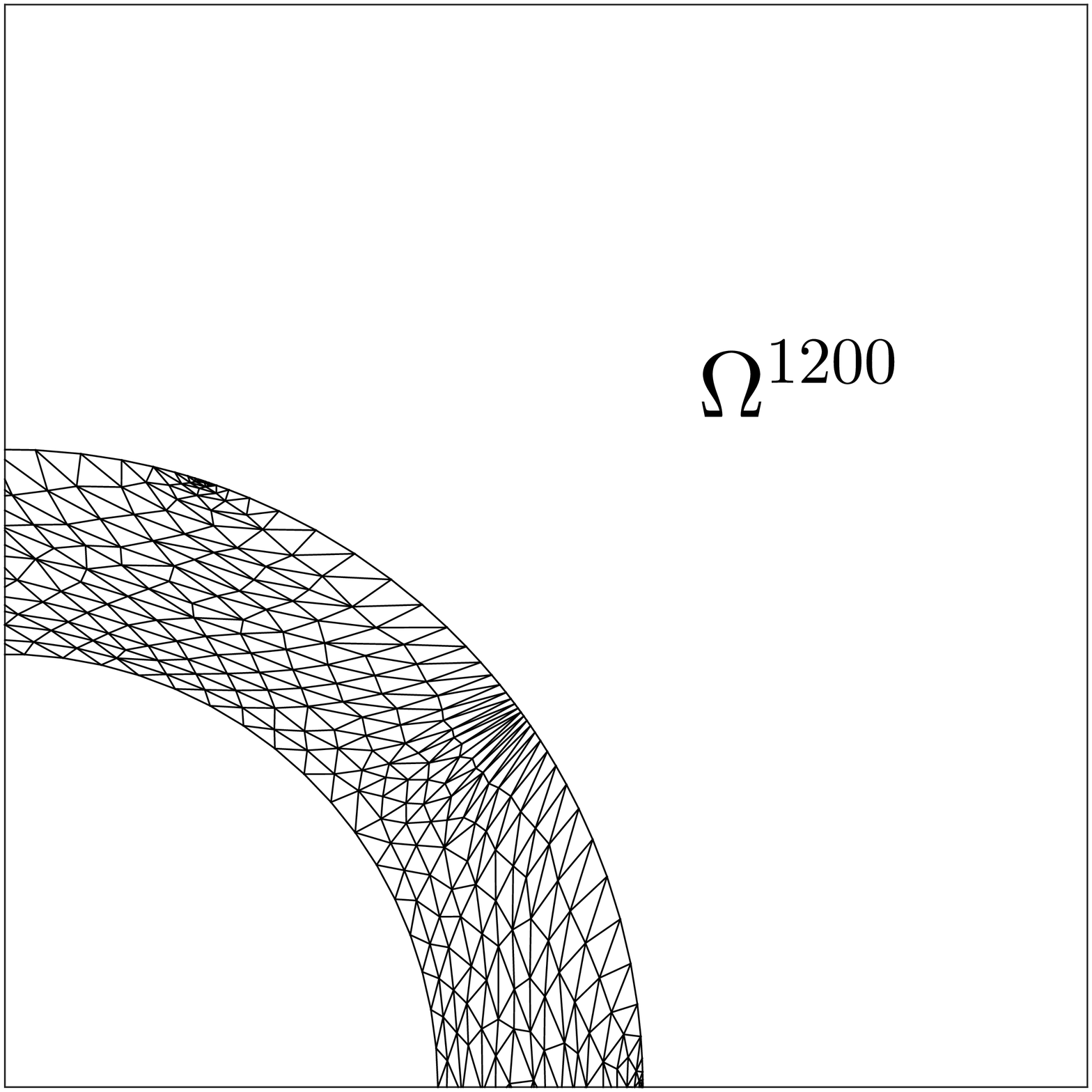}}
                \resizebox{0.445\textwidth}{!}{\includegraphics{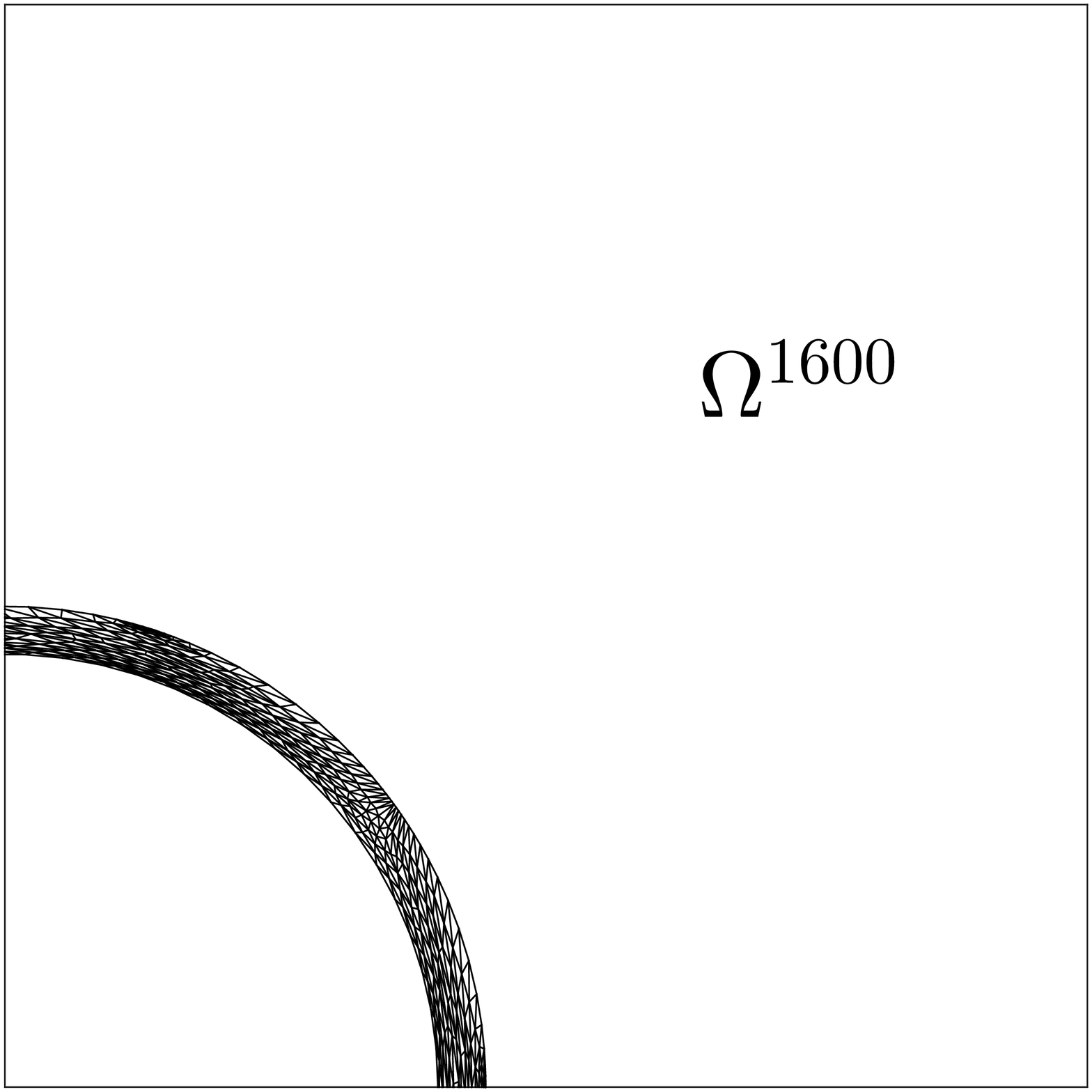}}
                 \caption{Mesh profile at $k=400, 800, 1200, 1600$}
                \label{fig:Fig10b}
        \end{subfigure}%
        \par\bigskip
        \begin{subfigure}[b]{0.5\textwidth}
                \centering
                \resizebox{0.9\textwidth}{!}{\includegraphics{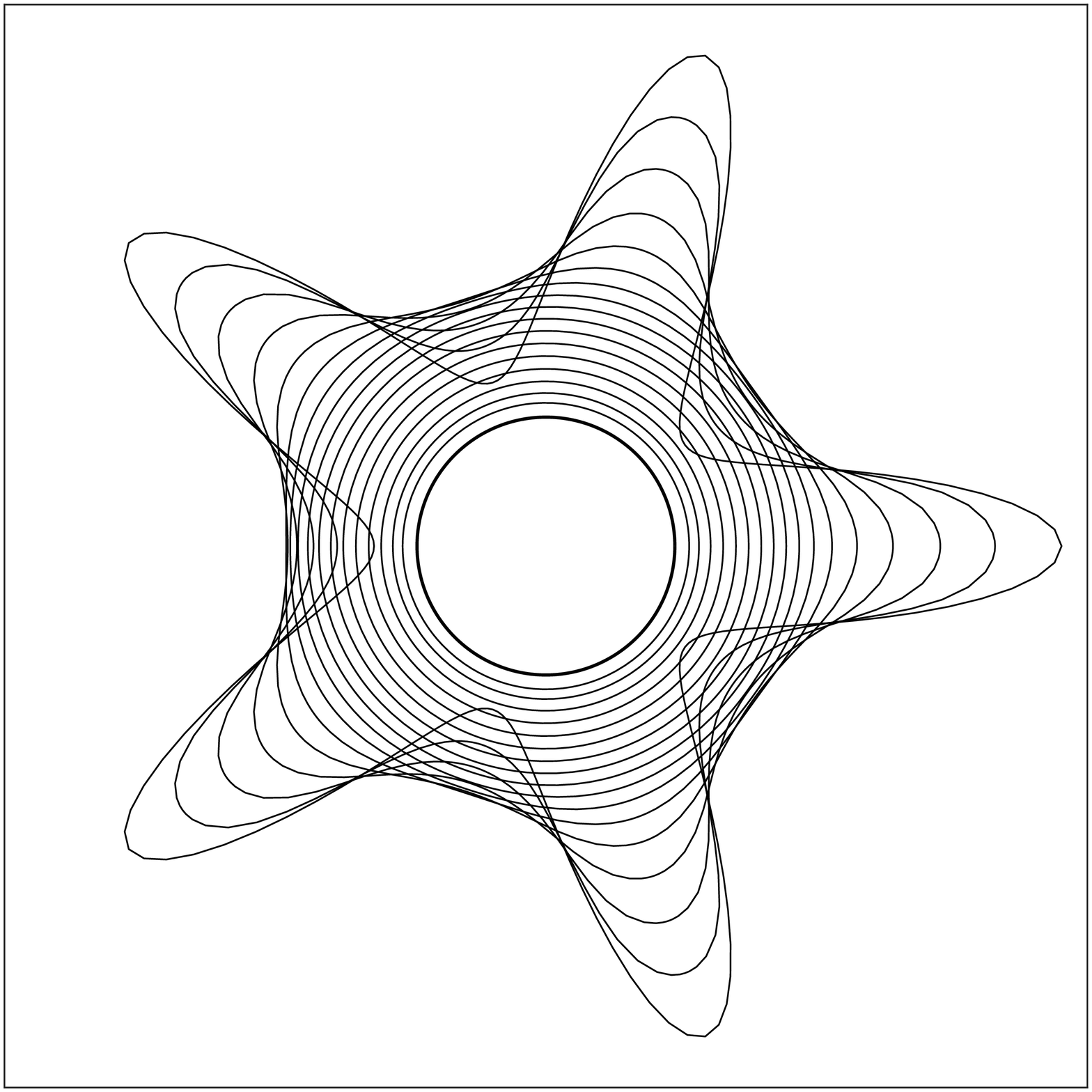}}
                \caption{Time evolution of the moving boundary}
                \label{fig:Fig10c}
        \end{subfigure}%
        \hfill
        \begin{subfigure}[b]{0.5\textwidth}
                \centering
                \resizebox{0.9\textwidth}{!}{\includegraphics{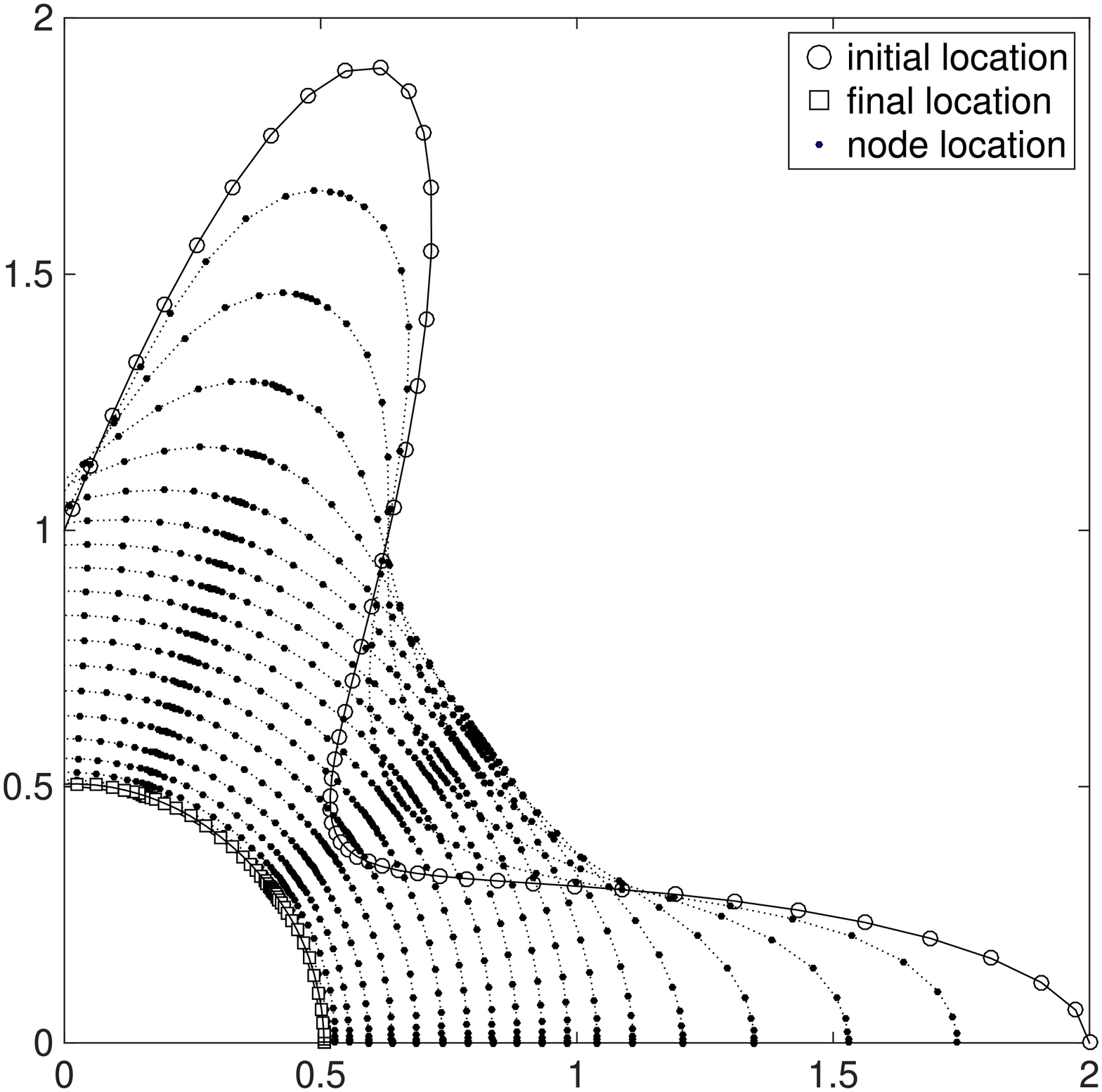}}
                \caption{Boundary nodes location at selected $k$s}
                \label{fig:Fig10d}
        \end{subfigure}%
\caption{Computational results for Example \ref{example4}}
\label{fig:Fig10}
\end{figure}
\subsection{EOC of CMM for the mean curvature flow problem}
\label{sec:EOC-CF}

We also check the accuracy of CMM for curvature flows in the same way as in subsection \ref{sec:EOC-HS}.
That is, we construct a manufactured solution and then compare the numerical solution obtained through the proposed scheme.
For this purpose, we state the following construction of the appropriate manufactured solution.
\begin{proposition}\label{prop:MS-CF}
  We suppose $\phi(x,t)$ is a smooth function with $\phi<0$ for $x \in \overline{B}$ and $|\nabla \phi| \neq 0$ on $\{\phi = 0\}$ for $t \in [0,T]$.
  We define $g:\mathbb{R}^d \times [0,T] \to \mathbb{R}$ as $g:= -\frac{\phi _t}{|\nabla \phi|} + \frac{\Delta \phi}{|\nabla \phi|} - \frac{((D^2 \phi)\nabla \phi) \cdot \nabla \phi}{|\nabla \phi|^3} $,
  and $\Omega_0:=\{ \phi(x,0)<0 \}$.
  Then, the moving domain $\Omega (t):= \{ x \in \mathbb{R}^d \mid \phi(x,t) < 0) \}$ satisfy $V_n = - \kappa  + g$ on $\Gamma(t)$, $t \in [0,T]$, and $\Omega(0)=\Omega_0$.
\end{proposition}
\begin{proof}
The proposition easily follows from straightforward computation of $V_n$, ${\nu}$, and the mean curvature $\kappa$ in terms of the level set function $\phi$.
\end{proof}
We now examine the EOC of the scheme when applied to solving the mean curvature problem using Proposition \ref{prop:MS-CF}.
In this experiment, the domains are initially discretized with uniform mesh size of width $h \approx 100 \times \tau$, and we set $\tau = 1/(100 \cdot 2^m)$, where $m=0,1,\ldots,5$.
The results are depicted in Fig. \ref{fig:Fig11}.
We observe from Fig. \ref{fig:Fig11a} an EOC of order one for $\tau$ against the boundary error ${\rm err}_{\Gamma}$.
On the other hand, it seems that, for $h \approx 100 \times \tau$, we only have a sub-linear order for the EOC with respect to $\varepsilon$ against ${\rm err}_{\Gamma}$.
In fact, the plot shown in Fig. \ref{fig:Fig11b} shows that the behavior due to the change of $\varepsilon$ is similar to Fig. \ref{fig:Fig5b}.
This is because CMM is an explicit method, and since the right side of the variational problem \eqref{curvatureflow_weakform} contains the mean curvature term which is a second derivative, then the time step size $\tau$ must be well less than $h$ to stabilize the numerical calculation.
In relation to this, notice in Fig. \ref{fig:Fig11a} that there is no corresponding error value for $\tau = 1/(100 \cdot 2^5)$ in case of $\varepsilon = 10^{-3}$.
This is because the scheme is becoming unstable after several time steps under this set of parameter values, causing the algorithm to stop.
So, for these reasons, we perform another experiment where $h \approx 200 \times \tau$ and consider different values for $\tau$.
The results are summarized in Fig. \ref{fig:Fig12} where we now observe an almost linear convergence behavior of the scheme with respect to $\varepsilon$ against the boundary ${\rm err}_{\Gamma}$ as conspicuous in Fig. \ref{fig:Fig12b}.
However, for small times steps, ${\rm err}_{\Gamma}$ is already saturated for $\tau$ of magnitude around or less than $10^{-3}$ as evident in Fig. \ref{fig:Fig12a}.
Nevertheless, the error values became smaller, which implies that the numerical solution is improved by taking sufficiently small time steps.
\begin{figure}[htbp]
\centering
        \begin{subfigure}[b]{0.49\textwidth}
                \centering
                \resizebox{\textwidth}{!}{\includegraphics{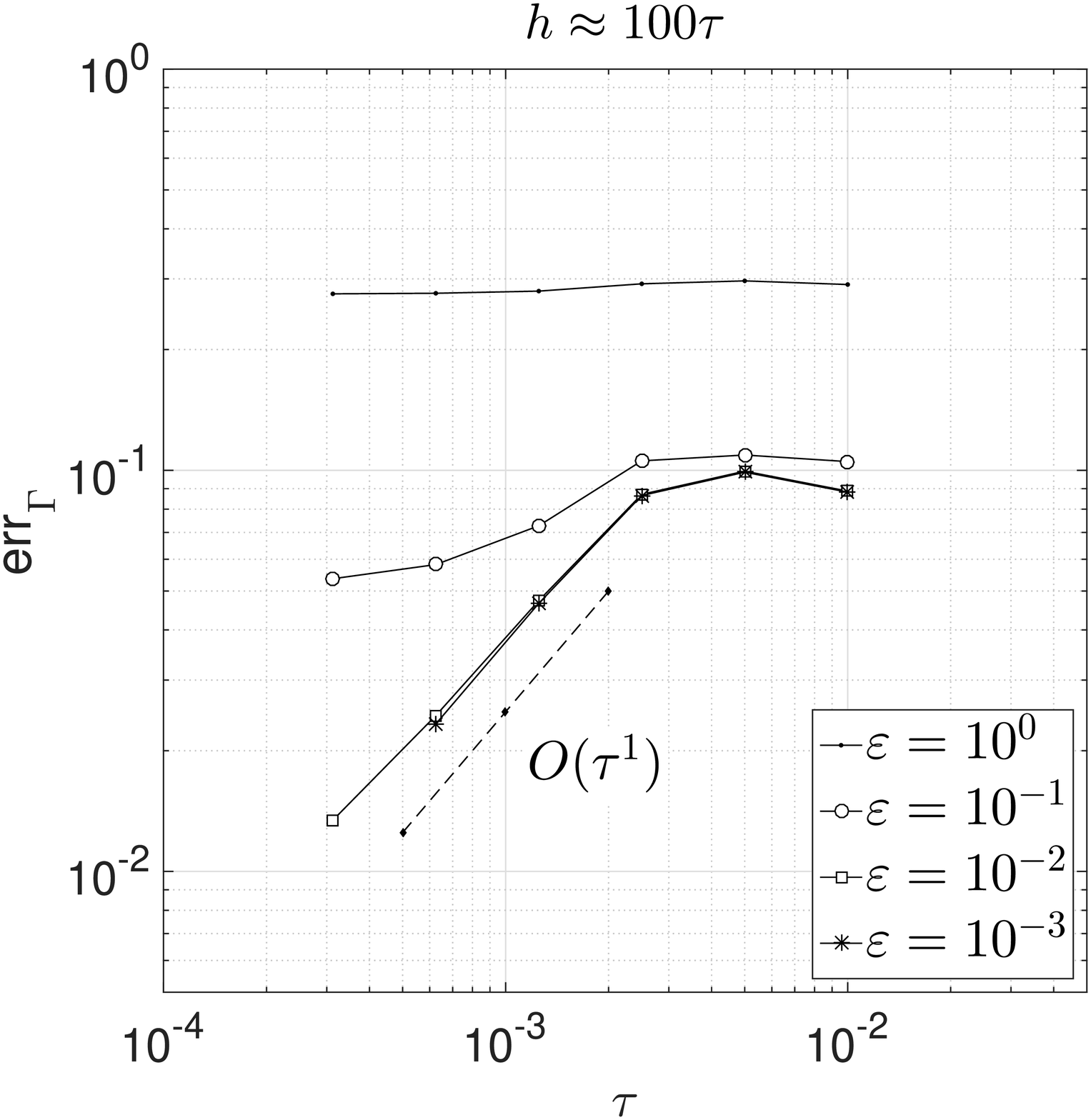}}
                \caption{$\tau$ vs $\operatorname{err}_{\Gamma}$ }
                \label{fig:Fig11a}
        \end{subfigure}%
        \hfill
	 \begin{subfigure}[b]{0.49\textwidth}
                \centering
                \resizebox{\textwidth}{!}{\includegraphics{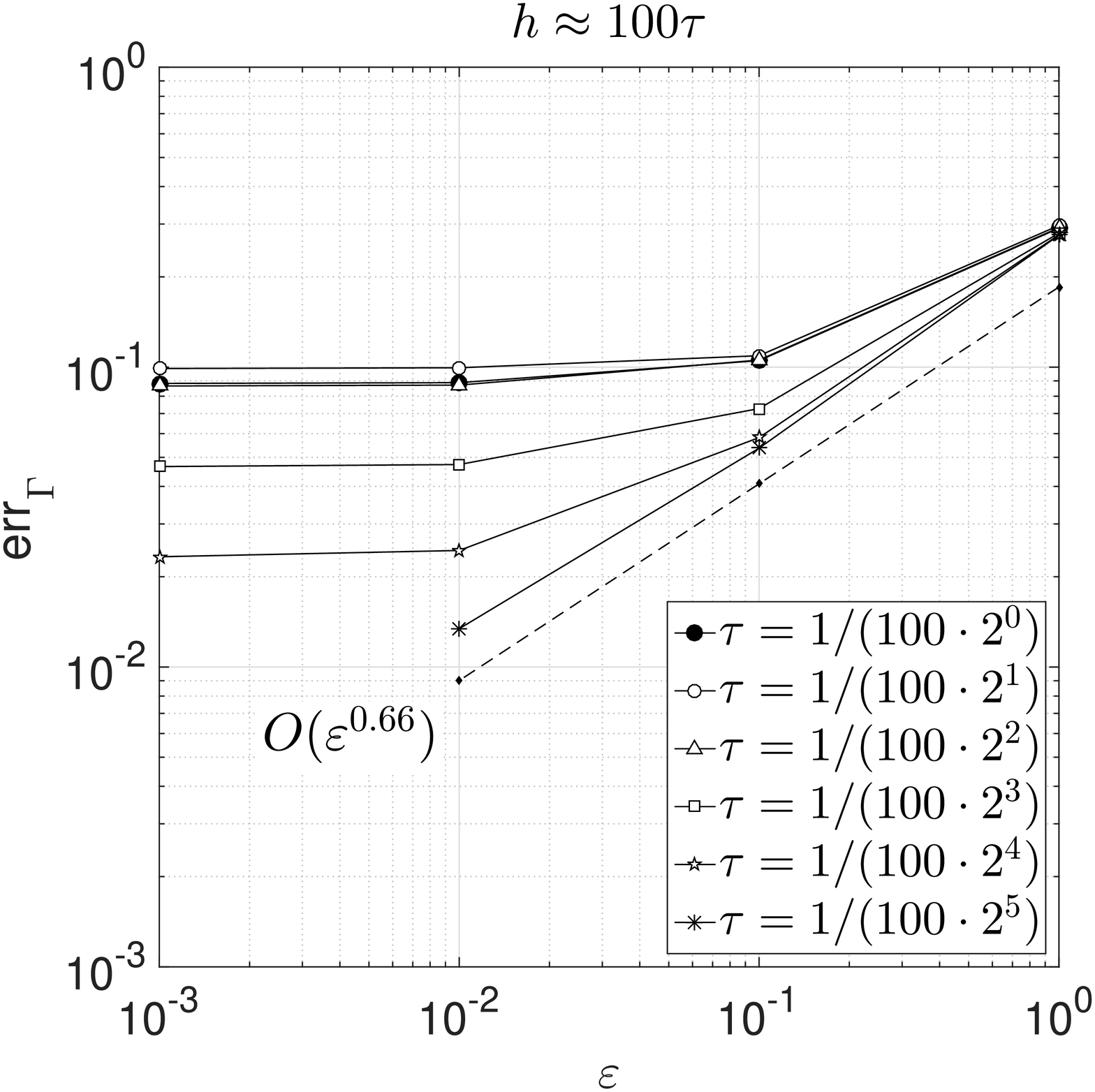}}
                \caption{$\varepsilon$ vs $\operatorname{err}_{\Gamma}$ }
                \label{fig:Fig11b}
        \end{subfigure}%
\caption{Error of convergences when $h \approx 100 \times \tau$}
\label{fig:Fig11}
\end{figure}
\vskip-15pt
\begin{figure}[htbp]
\centering
        \begin{subfigure}[b]{0.49\textwidth}
                \centering
                \resizebox{\textwidth}{!}{\includegraphics{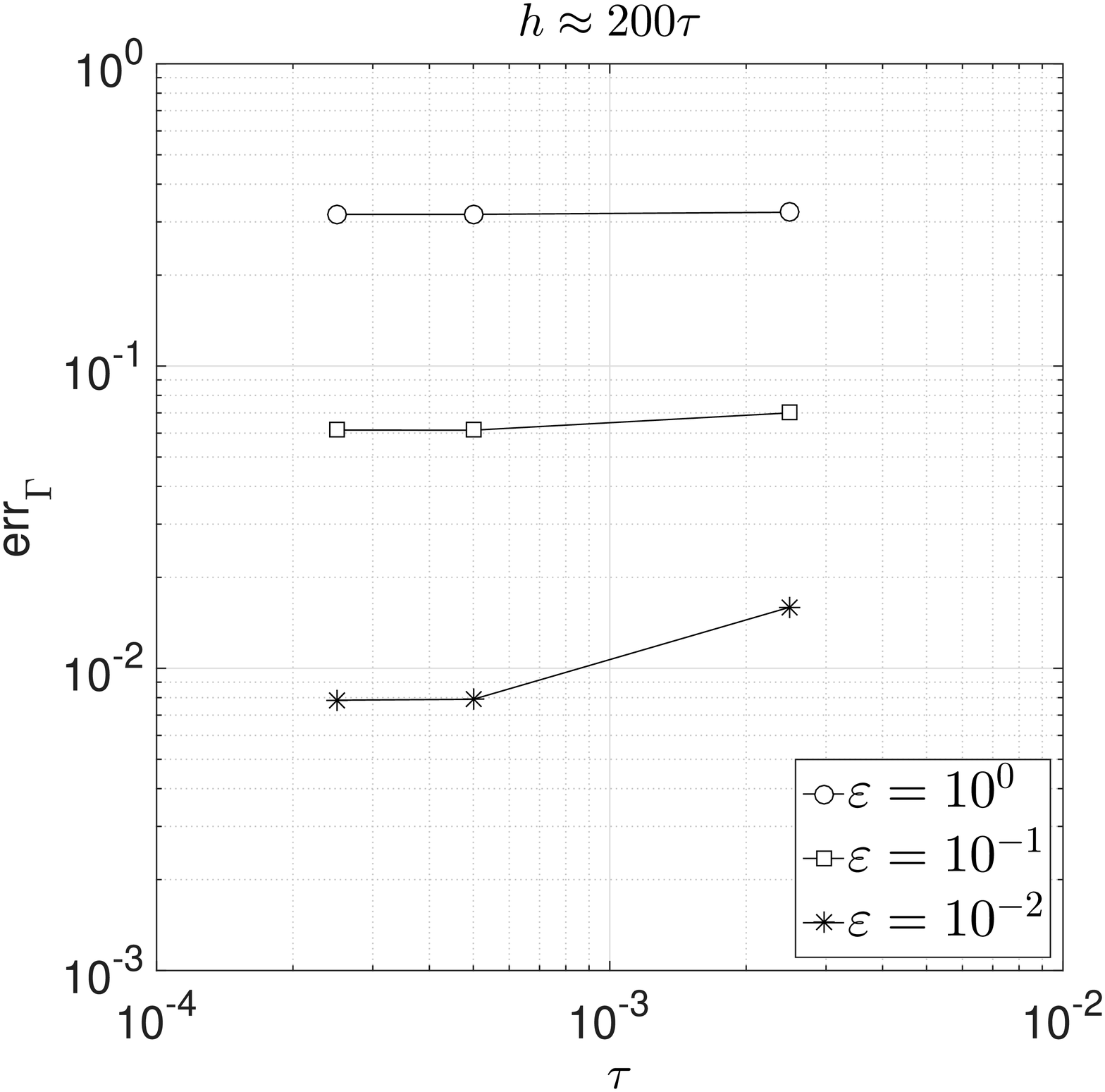}}
                \caption{$\tau$ vs $\operatorname{err}_{\Gamma}$ }
                \label{fig:Fig12a}
        \end{subfigure}%
        \hfill
	 \begin{subfigure}[b]{0.49\textwidth}
                \centering
                \resizebox{\textwidth}{!}{\includegraphics{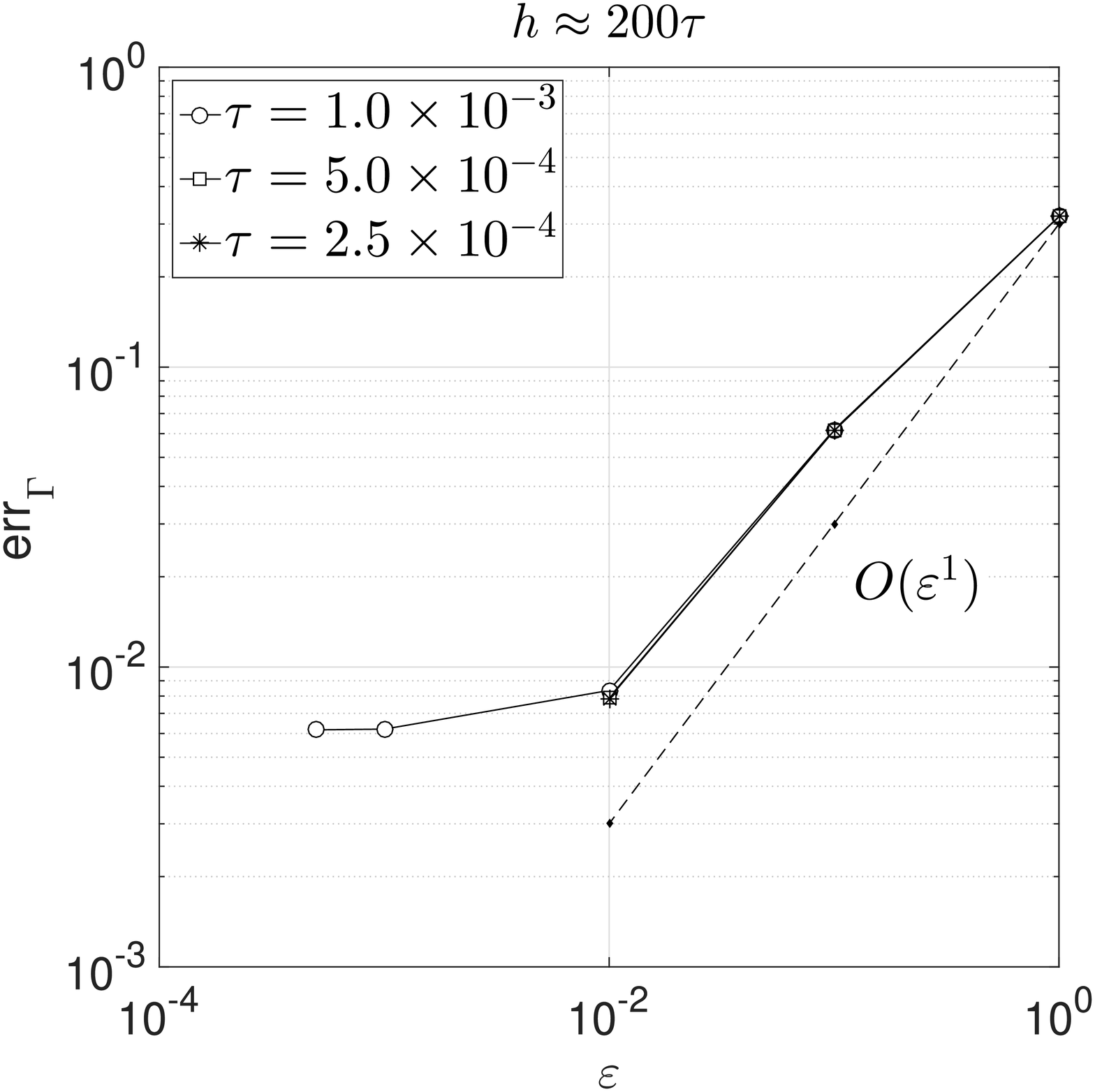}}
                \caption{$\varepsilon$ vs $\operatorname{err}_{\Gamma}$ }
                \label{fig:Fig12b}
        \end{subfigure}%
\caption{Error of convergences when $h \approx 200 \times \tau$}
\label{fig:Fig12}
\end{figure}
\section{Some Qualitative Properties of CMM}
\label{sec:properties}
In this section, we state and prove two simple properties of CMM related to the convergence to a stationary point $\Gamma^\ast$ of the moving boundary $\Gamma(t)$, $t \geqslant 0$, under a general description of the normal flow $V_{n}$, and another property we call the $\varepsilon$-approximation property of CMM.
In relation to the former result, let us consider the following abstract autonomous moving boundary problem.
\begin{prob}
\label{prob:abstractMBP}
	Given the initial profile $\Gamma_0$ and a real-valued function $F(\,\cdot\,;\Gamma):\Gamma \to \mathbb{R}$, find a moving surface $\Gamma(t)$, $t\geqslant 0$, which satisfies
	\begin{equation}
	\label{eq:abstractMBP}
	\left\{\arraycolsep=1.4pt\def\arraystretch{1.1}
	\begin{array}{rcll}
		V_{n}(x,t)		&=& F(x;\Gamma(t)),		&\quad x \in \Gamma(t), \quad t \geqslant 0,\\[0.5em]
		\Gamma(0)	&=& \Gamma_0.
	\end{array}
	\right.
	\end{equation}
\end{prob}
The particular forms of $F(x;\Gamma(t))$ that are of interest here are as follows:
  \begin{itemize}
    \item $F(x;\Gamma(t)) = (-\nabla u + \bb{\gamma}) \cdot {\nu} + \lambda$ in \eqref{eq:general_Hele-Shaw};
    \vspace{2pt}
    \item $F(x;\Gamma(t)) = - \kappa$ in \eqref{eq:curvature flow}.
\end{itemize}

Next, we define a stationary solution to Problem \ref{prob:abstractMBP}.
\begin{dfn}
  A domain $\Omega^*$ is said to be a \textit{stationary solution} to Problem \ref{prob:abstractMBP} if $\Gamma^* = \partial \Omega^*$, and $F(x;\Gamma^*) = 0$ for almost every $x \in \Gamma^*$.
\end{dfn}

Then, we associate with Problem \ref{prob:abstractMBP} the $\varepsilon$-regularized moving boundary problem given as follows:
\begin{prob}
\label{prob:epsilon_regularized}
	Let $B$ and $\Omega$ be two bounded domains with respective Lipschitz boundary $\partial B$ and $\Gamma:=\partial\Omega$ such that $\overline{B} \subset \Omega$.
	Given the initial profile $\Gamma_0$, a real-valued function $F(\,\cdot\, ;\Gamma) \in L^2(\Gamma)$, and a fix number $\varepsilon > 0$, we seek to find a moving surface $\Gamma(t)$, which satisfies
	\begin{equation}
	\label{eq:epsilon_regularized}
	\left\{\arraycolsep=1.4pt\def\arraystretch{1.}
	\begin{array}{rcll}
		- \Delta {\bb{w}} 	&=& \bb{0}	&\quad \text{in $\Omega(t) \setminus \overline{B}$,\quad $t \geqslant 0$},\\
		{\bb{w}} 			&=& \bb{0}	&\quad \text{on $\partial B$},\\
 		\varepsilon \nabla \bb{w} \cdot \nu + {\bb{w}} &=& F(\,\cdot\, ;\Gamma(t)) {\nu}		&\quad \text{on $\Gamma(t)$,\quad $t \geqslant 0$},\\
		V_{n} 			&=& \bb{w} \cdot \nu		&\quad \text{on $\Gamma(t)$,\quad $t \geqslant 0$}.
	\end{array}
	\right.
	\end{equation}
\end{prob}
With respect to Problem \ref{prob:epsilon_regularized}, a stationary solution $\Omega^*$ is define as follows.
\begin{dfn}
  A domain $\Omega^*$ is said to be a \textit{stationary solution} to Problem \ref{prob:epsilon_regularized} if $\Gamma^* = \partial \Omega^*$, and ${\bb{w}} \in H^1_{\partial B, \bb{0}}(\Omega^\ast\setminus \overline{B};\mathbb{R}^d)$ satisfies the variational equation.

  \begin{align}
    &\displaystyle \varepsilon \int_{\Omega^\ast \setminus \overline{B}} \nabla {\bb{w}} : \nabla \bb{\varphi} \ {\rm d}x
         + \int_{\Gamma^\ast}  {\bb{w}} \cdot \bb{\varphi}\ {\rm d}s \nonumber\\
      &\displaystyle \hspace{0.75in} = \int_{\Gamma^\ast} F(\cdot;\Gamma ^*) {\nu} \cdot \bb{\varphi}\ {\rm d}s,
            \quad \forall \bb{\varphi} \in H_{\partial B, \bb{0}}^1(\Omega^\ast \setminus \overline{B};\mathbb{R}^d),\label{cmm_weakform}\\
      & \text{and}\qquad \bb{w} \cdot \nu = 0 \quad \text{on $\Gamma^\ast$}.\label{cmm_bc}
  \end{align}
\end{dfn}

For Lipschitz domain $\Omega^\ast\setminus \overline{B}$ and $F(\,\cdot\, ;\Gamma) \in L^2(\Gamma)$, the variational problem \eqref{cmm_weakform} can be shown to have a weak solution ${\bb{w}} \in H^1(\Omega^\ast\setminus \overline{B};\mathbb{R}^d)$ via Lax-Milgram lemma.
With the above definition of a stationary point, we now state and prove our first result.
\begin{proposition}
\label{prop:convergence_to_a_stationary_point}
  We suppose $\Omega^* \supset \overline{B}$, $\Gamma^* = \partial \Omega^*$ is Lipschitz, and $F(\,\cdot\, ;\Gamma) \in L^2(\Gamma)$.
  Then, the following conditions are equivalent:
  \begin{enumerate}
    \renewcommand{\labelenumi}{(\roman{enumi})}
    \item $\Omega^*$ is a stationary solution to Problem \ref{prob:abstractMBP},
    \item $\Omega^*$ is stationary solution to Problem \ref{prob:epsilon_regularized}, for any $\varepsilon >0$,
    \item $\Omega^*$ is stationary solution to Problem \ref{prob:epsilon_regularized}, for some $\varepsilon >0$.
  \end{enumerate}
\end{proposition}
\begin{proof}
	Consider equation \eqref{eq:epsilon_regularized} over the stationary shape $\overline{\Omega^\ast}$ with Lipschitz boundary $\Gamma^\ast$.
	For the implication $(i) \Rightarrow (ii)$, we assume that $L^2(\Gamma) \ni F(\,\cdot\, ;\Gamma) = 0$, and we need to show that $\bb{w} \cdot \nu = 0$ on $\Gamma^\ast$.
	To do this, we apply integration by parts to \eqref{cmm_weakform}, and note that $\bb{w} = \bb{0}$ on $\partial B$, to obtain
	\begin{align*}
		0 \leqslant
		\int_{\Gamma^\ast} |\bb{w}|^2 \ {\rm d}s
		= -  \varepsilon  \int_{\Gamma^\ast} \ddn{\bb{w}} \cdot \bb{w} \, {\rm d}s
		= - \varepsilon  \int_{\Omega^\ast \setminus \overline{B}} |\nabla \bb{w}|^2 \, {\rm d}x
		\leqslant 0.
	\end{align*}
	Evidently, $\bb{w} \equiv \bb{0}$ on $\overline{\Omega^{\ast}}$, and in particular, $\bb{w}\cdot {\nu} = 0$ on $\Gamma^\ast$.

  	The proof of the direction $(ii) \Rightarrow (iii)$ is trivial.
	Finally, for the implication $(ii) \Rightarrow (iii)$, we need to prove that if $\bb{w} \cdot \nu = 0$ on $\Gamma^\ast$, where $\bb{w}$ satisfies the system \eqref{eq:epsilon_regularized} on $\overline{\Omega^\ast} \setminus B$, then $F = 0$ on $\Gamma^\ast$.
	In \eqref{cmm_weakform}, we take $\bb{\varphi} = \bb{w} \in H_{\partial B, \bb{0}}^1(\Omega^\ast \setminus \overline{B};\mathbb{R}^d)$ so that we get
	\[
   	\varepsilon \int_{\Omega^\ast \setminus \overline{B}} \nabla {\bb{w}} : \nabla \bb{w} \ {\rm d}x
    		 + \int_{\Gamma^\ast}  |{\bb{w}}|^2\ {\rm d}s \nonumber\\
			= \int_{\Gamma^\ast} F(\,\cdot\, ;\Gamma^*) {\nu}\, \cdot\, \bb{w}\ {\rm d}s = 0.
	\]
	This implies, obviously, that $\bb{w} \equiv \bb{0}$ on $\overline{\Omega^{\ast}}$.
	Going back to \eqref{cmm_weakform}, we see that $\int_{\Gamma^\ast} F(\,\cdot\, ;\Gamma^*) {\nu} \cdot \bb{\varphi}\ {\rm d}s = 0$, for all $\bb{\varphi} \in H_{\partial B, \bb{0}}^1(\Omega^\ast \setminus \overline{B};\mathbb{R}^d)$, from which we conclude that $F = 0$ on $\Gamma^\ast$.
	This proves the assertion. 
\end{proof}
In the rest of this section, we want to prove what we call the $\varepsilon$-approximation property of CMM.
For this purpose, we again fix $\Omega$ and $B$ and suppose that $\Gamma$ and $\partial B$ are Lipschitz regular.
Given a function $\bb{g} : \Gamma \to \mathbb{R}^d$, our main concern is the convergence of its Robin approximation to an original Dirichlet boundary condition associated with the following Laplace equation with pure Dirichlet boundary condition:
	\begin{equation}
	\label{eq:origsystem}
		- \Delta {\bb{v}} 	= \bb{0}	\quad \text{in $\Omega \setminus \overline{B}$},\qquad
		{\bb{v}} 			= \bb{0}	\quad \text{on $\partial B$},\qquad
		{\bb{v}} 			= \bb{g}	\quad \text{on $\Gamma$}.
	\end{equation}
For a given data $\bb{g} \in H^{1/2}(\Gamma; \mathbb{R}^d)$ and Lipschitz domain $\Omega \setminus \overline{B}$, it can be shown via Lax-Milgram lemma that the corresponding variational equation of \eqref{eq:origsystem} admits a unique weak solution $\bb{v} \in H^1(\Omega \setminus \overline{B}; \mathbb{R}^d)$.

	Now, we consider system \eqref{eq:origsystem} and denote its solution, depending on $\bb{g} \in H^{1/2}(\Gamma;\mathbb{R}^d)$, by $\bb{v}^i:=\bb{v}(\bb{g}^i)$.
	Also, we define the Dirichlet-to-Neumann map $\Lambda: H^{1/2}(\Gamma;\mathbb{R}^d) \to H^{-1/2}(\Gamma;\mathbb{R}^d)$.
	Then, we have the following lemma whose proof is given in the Appendix.
	\begin{lemma}
	\label{lem:inner_product}
		The map $(\,\cdot\,,\,\cdot\,)_{\Lambda} :  H^{1/2}(\Gamma;\mathbb{R}^d) \times  H^{1/2}(\Gamma;\mathbb{R}^d) \to \mathbb{R}$ defined as $(\bb{g}^1, \bb{g}^2)_{\Lambda} := (\Lambda\bb{g}^1, \bb{g}^2)_{L^2(\Gamma;\mathbb{R}^d)}$, for $\bb{g}^1, \bb{g}^2 \in H^{1/2}(\Gamma;\mathbb{R}^d)$, is an inner product on $H^{1/2}(\Gamma;\mathbb{R}^d)$, and is equivalent to the usual norm on $H^{1/2}(\Gamma;\mathbb{R}^d)$.
	\end{lemma}
	Now, for $\varepsilon > 0$ and $\bb{g}\in H^{1/2}(\Gamma;\mathbb{R}^d)$, we define $\bb{g}_{\varepsilon}$ such that
		$\varepsilon \Lambda \bb{g}_{\varepsilon}  +  \bb{g}_{\varepsilon} = \varepsilon \nabla {\bb{v}_{\varepsilon}} \cdot \nu + \bb{v}_{\varepsilon}  =: \bb{g}$,
	and consider the boundary value problem \eqref{eq:origsystem} with $\bb{v}$ and $\bb{g}$ replaced by $\bb{v}_{\varepsilon}$ and $\bb{g}_{\varepsilon}$, respectively, and, instead of the Dirichlet condition, we imposed on $\Gamma$ the Robin condition $\varepsilon \nabla {\bb{v}_{\varepsilon}} \cdot \nu + \bb{v}_{\varepsilon} = \bb{g}$.
	More precisely, we consider the mixed Dirichlet-Robin problem
	\begin{equation}
	\label{eq:epssystem}
		- \Delta {\bb{v}_{\varepsilon}} 	= \bb{0}	\quad \text{in $\Omega \setminus \overline{B}$},\quad
		{\bb{v}_{\varepsilon}} 		= \bb{0}	\quad \text{on $\partial B$},\quad
		\varepsilon \nabla \bb{v}_{\varepsilon} \cdot \nu + \bb{v}_{\varepsilon} = \bb{g} 	\quad \text{on $\Gamma$}.
	\end{equation}
	Let us define the bilinear form $a^\varepsilon(\,\cdot\,, \, \cdot \,)$ as follows:
	\[
		a^\varepsilon(\bb{\varphi}, \bb{\psi}):= \ _{H^{-1/2}}\langle (\varepsilon \Lambda + \bb{I})\bb{\varphi}, \bb{\psi} \rangle_{H^{1/2}}
			= \varepsilon(\bb{\varphi}, \bb{\psi})_{\Lambda} + (\bb{\varphi}, \bb{\psi} )_{L^2(\Gamma;\mathbb{R}^d)}.
	\]
	Then, we may write a weak formulation on $\Gamma$ for $\bb{g}^\varepsilon$ as follows:
	find $\bb{g}^\varepsilon \in H^{1/2}(\Gamma;\mathbb{R}^d)$ such that
	\begin{equation}
	\label{eq:weak_form_on_gamma}
		a^\varepsilon(\bb{g}^\varepsilon, \bb{\varphi}) = (\bb{g}, \bb{\varphi})_{L^2(\Gamma;\mathbb{R}^d)},\qquad \text{for all $\bb{\varphi} \in H^{1/2}(\Gamma;\mathbb{R}^d)$}.
	\end{equation}
	Again, the existence of unique weak solution $\bb{g}^\varepsilon\in H^{1/2}(\Gamma;\mathbb{R}^d)$ to the above variational problem can be proven using Lax-Milgram lemma.

We now exhibit our second convergence result in the following proposition which simply states the convergence of the Robin approximation to the original Dirichlet data in $L^2(\Gamma)$ sense as the parameter $\varepsilon$ goes to zero provided that the Neumann data $\Lambda \bb{g}$ is square integrable.
\begin{proposition}
\label{prop:epsilon_approximation}
Let $\bb{g} \in  H^{1/2}(\Gamma;\mathbb{R}^d)$ and $\Gamma$ be Lipschitz regular.
If $\Lambda \bb{g} \in L^2(\Gamma;\mathbb{R}^d)$, then the following estimate holds
$\left\| \bb{g}^\varepsilon -  \bb{g} \right\|_{L^2(\Gamma;\mathbb{R}^d)}
			\leqslant \varepsilon \left\| \Lambda \bb{g} \right\|_{L^2(\Gamma;\mathbb{R}^d)}$.
\end{proposition}
\begin{proof}
	Taking the test function in \eqref{eq:weak_form_on_gamma} as $\bb{\varphi}:= \bb{g}^\varepsilon -  \bb{g} \in H^{1/2}(\Gamma;\mathbb{R}^d)$ gives us the following sequence of equations:
		$a^\varepsilon(\bb{g}^\varepsilon -  \bb{g}, \bb{g}^\varepsilon -  \bb{g})
		= (\bb{g}, \bb{g}^\varepsilon -  \bb{g})_{L^2(\Gamma;\mathbb{R}^d)}
				- \varepsilon (\bb{g}, \bb{g}^\varepsilon -  \bb{g})_{\Lambda}
					- (\bb{g}, \bb{g}^\varepsilon -  \bb{g})_{L^2(\Gamma;\mathbb{R}^d)}
		= - \varepsilon (\bb{g}, \bb{g}^\varepsilon -  \bb{g})_{\Lambda}$.
	This gives us the estimate
	$\left\| \bb{g}^\varepsilon -  \bb{g} \right\|^2_{L^2(\Gamma;\mathbb{R}^d)}
			\leqslant - \varepsilon (\bb{g}, \bb{g}^\varepsilon -  \bb{g})_{\Lambda}$.
	Furthermore, if $\Lambda \bb{g} \in L^2(\Gamma;\mathbb{R}^d)$, then we can write this inequality as
	$\left\| \bb{g}^\varepsilon -  \bb{g} \right\|_{L^2(\Gamma;\mathbb{R}^d)}
			\leqslant \varepsilon \left\| \Lambda \bb{g} \right\|_{L^2(\Gamma;\mathbb{R}^d)}$, as desired.
\end{proof}
\section{Conclusion}
\label{sec:conclusion}
We have developed a finite element scheme we called the `comoving mesh method' or CMM for solving certain families of moving boundary problems.
We applied the proposed scheme in solving the classical Hele-Shaw problem and the exterior Bernoulli free boundary problem.
In the latter case, we found that the generalized Hele-Shaw problem with normal velocity flow $V_{n} = -\nabla u \cdot {\nu} + \lambda$, where $\lambda <0$ converges to a stationary point which coincides with the optimal shape solution of the said free boundary problem.
We have also demonstrated the applicability of CMM in solving a moving boundary problem involving the mean curvature flow equation $V_{n} = -\kappa$.
The numerical experiments performed here showed that the experimental order of convergence of the approximate solutions obtained using CMM are mostly linear for both the Hele-Shaw problem and the mean curvature problem.
In case of the former problem, this linear order of convergence was seen for time step sizes that is as large as the mesh size value.
On the other hand, for the mean curvature problem, it was observed that the magnitude of the time step-size has to be well less than the width of the mesh in order for the numerical scheme to be stable and obtained a (nearly) linear order of convergence with respect to the parameter $\varepsilon$ against the boundary shape error.
Finally, we have also presented two simple properties of CMM pertaining to its stationary solution and a convergence result regarding the $\varepsilon$-approximation of $V_{n}$.

In our next investigation, we will apply the method in solving more general moving boundary problems such as the Stefan problem and the two-phase Navier-Stokes equations.
Moreover, we want to treat the Gibbs-Thomson law which assumes the condition $u=\sigma \kappa$ on the moving boundary.

\appendix
\section{Proof of Lemma \ref{lem:inner_product}}
Let us now prove Lemma \ref{lem:inner_product}.
\begin{proof}
	Consider system \eqref{eq:origsystem} whose solution is given by $\bb{v}^i:=\bb{v}(\bb{g}^i)$.
	Also, consider the Dirichlet-to-Neumann map $\Lambda: H^{1/2}(\Gamma;\mathbb{R}^d) \to H^{-1/2}(\Gamma;\mathbb{R}^d)$.
	\sloppy Then, for $\bb{g}^1, \bb{g}^2 \in H^{1/2}(\Gamma;\mathbb{R}^d)$, the binary operation $(\bb{g}^1, \bb{g}^2)_{\Lambda} := (\Lambda\bb{g}^1, \bb{g}^2)_{L^2(\Gamma;\mathbb{R}^d)}$ is an inner product on $H^{1/2}(\Gamma;\mathbb{R}^d)$.
	Indeed, we have the following arguments
	\begin{description}
		\item[(i)] since, for any $\bb{g}^3 \in H^{1/2}(\Gamma;\mathbb{R}^d)$ and
			$c \in \mathbb{R}$, we have $(\Lambda (c \bb{g}^1 + \bb{g}^2), \bb{g}^3)_{L^2(\Gamma;\mathbb{R}^d)}
				= \int_{\Gamma} (c \nabla \bb{v}^1 + \nabla \bb{v}^2)\cdot \nu\, \bb{v}^3\, {\rm d}s
				= c (\Lambda \bb{g}^1, \bb{g}^3)_{L^2(\Gamma;\mathbb{R}^d)} + (\Lambda \bb{g}^2, \bb{g}^3)_{L^2(\Gamma;\mathbb{R}^d)}
			$, then $(\, \cdot \,, \, \cdot \,)_{\Lambda}$ is linear with respect to its first argument;
		\item[(ii)] the binary operation $(\, \cdot \,, \, \cdot \,)_{\Lambda}$ is positive definite because, for any $\bb{g}^2 \in H^{1/2}(\Gamma;\mathbb{R}^d)$, we have $(\Lambda \bb{g}, \bb{g})_{L^2(\Gamma;\mathbb{R}^d)} = \int_{\Gamma} (\nabla \bb{v} \cdot \nu) \bb{v}\ {\rm d}s = \int_{\overline{\Omega}\setminus B} |\nabla \bb{v}|^2  \ {\rm d}x \geqslant 0$;
		\item[(iii)] also, it is point-separating, that is $(\Lambda \bb{g}, \bb{g})_{L^2(\Gamma;\mathbb{R}^d)} = 0$ if and only if $\bb{g} \equiv \bb{0}$; and,
		\item[(iv)] lastly, the operation is symmetric because
			$(\Lambda \bb{g}^1, \bb{g}^2)_{L^2(\Gamma;\mathbb{R}^d)}
				= \int_{\Gamma} (\nabla \bb{v}^1 \cdot \nu) \bb{v}^2\ {\rm d}s
					= \int_{\overline{\Omega}\setminus B} \nabla \bb{v}^1 : \nabla \bb{v}^2  \ {\rm d}x
						= \int_{\Gamma} \bb{v}^1 (\nabla \bb{v}^2 \cdot \nu) \ {\rm d}s
							= ( \bb{g}^1, \Lambda \bb{g}^2)_{L^2(\Gamma;\mathbb{R}^d)}$.
	\end{description}
	Additionally, for Lipschitz $\Gamma$, the inner product $(\, \cdot \,, \, \cdot \,)_{\Lambda}$ is equivalent to the natural one in $H^{1/2}(\Gamma;\mathbb{R}^d)$.
	Here, $H^{1/2}(\Gamma;\mathbb{R}^d)$ is viewed as the image of the trace operator $\gamma_{\Gamma}$ on $\Gamma$ (i.e., $\operatorname{Im}(\gamma_{\Gamma}) = \gamma_{\Gamma}(H^1(\Omega;\mathbb{R}^d))$).
	Consequently, by Riesz representation theorem, together with the embedding $H^{-1/2}(\Gamma;\mathbb{R}^d) \supset L^2(\Gamma;\mathbb{R}^d) \supset H^{1/2}(\Gamma;\mathbb{R}^d)$, we conclude that $\Lambda \in \operatorname{Isom}(H^{1/2}(\Gamma;\mathbb{R}^d),H^{-1/2}(\Gamma;\mathbb{R}^d))$.
	This proves the lemma.
\end{proof}



\end{document}